\documentclass[11pt]{amsart}
\usepackage[english]{babel}
\usepackage{fullpage}
\usepackage{amssymb,pstricks,amscd,epsfig}
\usepackage{graphicx}

\usepackage{amsmath, amsthm, amssymb}
\usepackage{pdfsync}
\usepackage[latin1]{inputenc}
\usepackage{stmaryrd}
\usepackage{color}
\usepackage{setspace}
\usepackage[matrix,arrow,tips,curve]{xy}

\usepackage{hyperref}

\makeatletter
 \renewcommand\section{\@startsection {section}{1}{\z@}%
     {-4.5ex \@plus -1ex \@minus -.2ex}%
     {2.3ex \@plus.8ex}%
    {\centering\scshape}}

\usepackage{cite}

\newcommand{\Z}{\mathbb{Z}}
\newcommand{\Q}{\mathbb{Q}}
\newcommand{\R}{\mathbb{R}}
\newcommand{\C}{\mathbb{C}}
\renewcommand{\P}{\mathbb{P}}

\newcommand{\mc}{\mathcal}

\newcommand{\ov}{\overline}

\newcommand{\be}{\begin{equation}}
\newcommand{\ee}{\end{equation}}

\newcommand{\wt}{\widetilde}

\newcommand{\hk}{hyper--K\"ahler }

\newtheorem{thm}{Theorem}
\newtheorem{thmintro}{Theorem}
\newtheorem{prop}[thm]{Proposition}
\newtheorem{lemma}[thm]{Lemma}
\theoremstyle{definition}

\newtheorem{rem}[thm]{Remark}
\newtheorem{defin}[thm]{Definition}
\newtheorem{cor}[thm]{Corollary}

\newtheorem{example}[thm]{Example}

\newtheorem{conj}[thm]{Conjecture}

\numberwithin{thm}{section}

\newcommand{\bdg}{\begin{dg}}
\newcommand{\edg}{\end{dg}}

\numberwithin{equation}{section}

\DeclareMathOperator{\codim}{codim}

\DeclareMathOperator{\rk}{rk}
\DeclareMathOperator{\Hom}{Hom}

\DeclareMathOperator{\coker}{coker}

\DeclareMathOperator{\Sym}{Sym}

\DeclareMathOperator{\Supp}{Supp}

\DeclareMathOperator{\Pic}{Pic}

\DeclareMathOperator{\Jac}{Jac}
\DeclareMathOperator{\NS}{NS}

\DeclareMathOperator{\im}{Im}

\DeclareMathOperator{\Nef}{Nef}
\DeclareMathOperator{\Mov}{Mov}
\DeclareMathOperator{\Def}{Def}

\DeclareMathOperator{\CH}{CH}

\newcommand{\Jbmc}{{\bar{\mc J}}}

\newcommand{\Yp}{{\mc Y'}}
\newcommand{\fp}{{f'}}
\newcommand{\Mp}{{\mc M'}}

\newcommand{\Up}{{U'}}

\newcommand{\Pfd}{{(\P^5)^\vee}}

\author[G. Sacc\`a]{Giulia Sacc\`a, with an Appendix by Claire Voisin}
\address{Mathematics Department\\
Columbia University \\
Mathematics Department\\
2990 Broadway\\
New York, NY 10027}

\title[Birational geometry of the Intermediate Jacobian fibration]{Birational geometry of the Intermediate Jacobian fibration of a cubic fourfold}

\begin{document}

\maketitle

\begin{abstract}

We show that the intermediate Jacobian fibration associated to \emph{any} smooth cubic fourfold $X$ admits a \hk compactification $J(X)$ with a regular Lagrangian fibration $\pi: J \to \P^5$. This builds upon \cite{LSV}, where the result is proved for general $X$, as well as on the degeneration techniques introduced in \cite{KLSV} and the minimal model program. 
 We then study some aspects of the birational geometry of $J(X)$: for very general $X$ we compute the movable and nef cones of $J(X)$, showing that $J(X)$ is not birational to the twisted version of the intermediate Jacobian fibration \cite{Voisin-twisted}, nor to an OG$10$-type moduli space of objects in the Kuznetsov component of $X$; for any smooth $X$ we show, using normal functions, that the Mordell-Weil group $MW(\pi)$ of the fibration is isomorphic to the integral degree $4$ primitive algebraic cohomology of $X$, i.e., $MW(\pi) \cong H^{2,2}(X, \Z)_0$.

\end{abstract}

\tableofcontents

\section*{Introduction}

 The geometry of smooth cubic fourfolds has ties to that of K3 surfaces and, more generally, to that of higher dimensional \hk manifolds.  For example, with certain special cubic fourfolds one can associate a K3 surface via Hodge theoretic \cite{Hassett-special} or derived categorical \cite{Kuznetsov} methods. From a more geometric perspective, given  a smooth cubic fourfold $X$,  \hk manifolds of K3$^{[n]}$-type are constructed geometrically, via parameter spaces of rational curves of certain degrees on  $X$ \cite{Beauville-Donagi,LLSvS}, or as moduli spaces of objects in the Kuznetsov component of $X$ \cite{Lahoz-Lehn-Macri-Stellari,BLMNPS}. These constructions give rise to $20$-dimensional families of polarized \hk manifolds, the maximal possible dimension of families of polarized \hk manifolds of K3$^{[n]}$-type.
As the cubic fourfold becomes special, for example when it acquires more algebraic classes, the geometry of these \hk manifolds also becomes more interesting. For example, when $X$ has an associated K3 surface in the sense of \cite{Hassett-special, Kuznetsov, Addington-Thomas, Huybrechts-category}, these \hk manifolds become isomorphic, or birational, to moduli spaces of objects in the derived category of the corresponding K3 surface \cite{Addington,BLMNPS}. 

In \cite{LSV} a Lagrangian fibered \hk manifold is constructed starting from a general cubic fourfold. This \hk manifold is a deformation of O'Grady's $10$-dimensional exceptional example.
More precisely, let $X \subset \P^5$ be a smooth cubic fourfold and let $\pi_U: J_U \to U \subset  (\P^5)^\vee$ be the family of  intermediate Jacobians of the smooth hyperplane sections of $X$.  This fibration was considered by Donagi-Markman in \cite{Donagi-Markman}, where they showed that the total space has a holomorphic symplectic form. The main result of \cite{LSV} is to construct, for general $X$, a smooth projective \hk compactification $J$ of $J_U$, with a flat morphism $J \to \Pfd$ extending $\pi_U$, and to show that this \hk $10$-fold is deformation equivalent to O'Grady's $10$-dimensional example. In \cite{Voisin-twisted}, Voisin constructs a \hk compactification $J^T$ of a natural $J_U$-torsor $J_U^T$, which is non-trivial for very general $X$. The two \hk manifolds $J$ and $J^T$ are birational over countably many hyperpurfaces in the moduli space of cubic fourfolds.
These two constructions give rise to two $20$-dimensional families of \hk manifolds of OG$10$-type, each of which form an open subset of a codimension two locus inside the moduli space of \hk manifolds in this deformation class.

If one wishes to study the geometry of these \hk manifolds as the cubic fourfold becomes special, a first step is to check if a \hk compactification of the fibration $J_U \to U $ can be constructed for an arbitrary smooth cubic fourfold. The starting result of this paper is that this can indeed be done.

\begin{thmintro}[Theorem \ref{hkcom}] \label{main thm 1} 
Let $X \subset \P^5$ be a smooth cubic fourfold, and let $\pi_U: J_U \to U \subset  (\P^5)^\vee$ be the Donagi--Markman fibration. There exists a smooth projective hyper--K\"ahler compactification $J $ of $ J_U$ with a morphism $\pi: J \to (\P^5)^\vee$ extending $\pi_U$.
\end{thmintro}

 The same techniques also give the existence of a Lagrangian fibered \hk  compactification for the non trivial $J_U$-torsor $J_U^T \to U$  of \cite{Voisin-twisted} for any smooth $X$  (see Remark \ref{twisted}).  Moreover, with little extra work, the Theorem is proved also for mildly singular cubic fourfolds such as, for example, cubic fourfolds with a simple node (see Theorem \ref{Csix}). For a general cubic fourfolds with one node, the existence of such a Lagrangian fibered \hk manifold provides a positive answer to a question of Beauville \cite{Beauville-Fano-K3}(see Remark \ref{rem beau}).  
 
We should point out that as a consequence of the  ``finite monodromy implies smooth filling'' results of \cite{KLSV}, we prove in Proposition \ref{prop hk model} that $J_U$ admits projective birational model that is hyper-K\"ahler. Theorem \ref{main thm 1}  shows that there exists a \hk model with a Lagrangian fibration extending $\pi_U$.

There are several ingredients in the construction of the \hk compactification of \cite{LSV}: a cycle-theoretic construction of the holomorphic symplectic form, the problem of the existence of so--called very good lines for any hyperplane section of $X$,  a smoothness criterion for relative compactified Prym varieties, the independence of the compactification from the choice of a very good line.  Here we pursed a different direction and instead  rely on the existence of a \hk compactification for general $X$, use the degenerations techniques introduced in \cite{KLSV}, and implement some results from birational geometry and the minimal model program, following  \cite{Lai,Kollar-Elliptic}. One advantage of our method is that it opens the door to using birational geometry to compactify Lagrangian fibrations.

The second result of this paper is concerned with the \hk birational geometry of $J$. We show that the relative theta divisor $\Theta$  of the fibration is a prime exceptional divisor and that for general $X$ it can be contracted after a Mukai flop. 

\begin{thmintro}[Theorem \ref{NSU}]  \label{main thm 2} Let $q$ be the  Beauville-Bogomolov form on $H^2(J, \Z)$. The relative theta divisor $\Theta \subset J$ is a prime exceptional divisor with $q(\Theta)=-2$. 
For very general $X$, there is a unique other \hk birational model of $J$, denoted by $N$,  which is the Mukai flop $p: J \dashrightarrow N$ of $J$ along the image of the zero section.  $N$ admits a divisorial contraction $h: N \to \bar N$, which contracts the proper transform of $\Theta$ onto an $8$-dimensional variety which is birational to the LLSvS $8$-fold $Z(X)$.
\end{thmintro}

Thus, for very general $X$, $J$ is the unique \hk birational model with a Lagrangian fibration, it is not birational to $J^T$  (Corollary \ref{J not JT}), and its movable cone is the union of its nef cone and the nef cone of $N$. This answers a question by Voisin \cite{Voisin-twisted}. 
As a consequence of this Theorem we show that for very general $X$, $J$ is not birational to a moduli spaces of objects in the Kuznetsov component $\mc K u(X)$ of $X$ (see Corollary \ref{notkuz}). 
In the opposite direction, it was recently proved \cite{Pertusi-et-al}) that the twisted \hk manifold $J^T$ is birational to a moduli space of objects of OG$10$-type in $\mc K u(X)$. By objects of OG$10$-type, we mean objects whose Mukai vector is of the form $2w$, with $w^2=2$.  As a consequence, the family of intermediate Jacobian fibrations is the only known family of \hk manifolds associated with cubic fourfolds whose very general point cannot be described as a moduli space of objects in the Kuznetsov component of $X$.

Given $J=J(X)$, a \hk compactification of the intermediate Jacobian fibration for any smooth cubic fourfold  $X$, a natural question to ask is how the geometry of $J$ changes as $X$ becomes less general. One way to answer this question is the following theorem,  describing the Mordell-Weil group of $\pi$ in terms of the primitive algebraic cohomology of $X$. In Section \ref{section MW} we prove:

\begin{thmintro} \label{main thm 3} (Theorem \ref{MW}) 
 Let $MW(\pi)$ be the Mordell-Weil group of $\pi: J \to \P^5$, i.e., the group of rational sections of $\pi$ and let $H^{2,2}(X, \mathbb Z)_0$ be the primitive degree $4$ integral cohomology of $X$.
 The natural group homomorphism 
\[
\phi_X: H^{2,2}(X, \mathbb Z)_0 \to MW(\pi)
\]
induced by the Abel-Jacobi map is an isomorphism.\end{thmintro}

The proof of this result uses the theory of normal functions, as developed by Griffiths and Zucker, as well as the techniques used by Voisin to prove the integral Hodge conjecture for cubic fourfolds. 
A consequence of this is a  geometric description of the Lagrangian fibered \hk manifolds with maximal Mordell-Weil rank whose existence was proved by Oguiso in \cite{Oguiso-MW}: indeed, Oguiso's examples are (birationally) given by $J=J(X) \to \P^5$, where $X$ a smooth cubic fourfold with $H^{2,2}(X, \Z)$ of maximal rank.

\subsection*{Plan of the paper} In Section \ref{birational section} we prove the existence of a \hk compactification for $J_U$ and for $J_U^T$, in the case of any smooth, or mildly singular, $X$. This uses some results from the minimal model program, which are briefly recalled. 
In Section \ref{ogtentype} we review some basic results about moduli spaces of OG$10$-type and we compute, using the Bayer-Macr\`i techniques adapted to these singular moduli spaces by Meachan-Zhang \cite{Meachan-Zhang}, the nef and movable cones of certain moduli spaces of OG$10$-type  that appear as limits of the intermediate Jacobian fibration. 
The main result of Section \ref{section theta} is the computation that $q(\Theta)=-2$. 
 Section \ref{section general X} is devoted to the proof of Theorem \ref{main thm 2} and its preparation: Given a family of  cubic fourfold degenerating to the chordal cubic, we construct a certain degeneration of the intermediate Jacobian fibration and identify the limit of the corresponding degeneration of the relative Theta divisor. By the results of Section \ref{ogtentype}, the limiting theta divisor can be contracted after a Mukai flop of the zero section and we deduce the analogue result for $\Theta$.
The computation of the Mordell-Weil group occupies Section \ref{section MW}.

Finally, in the Appendix by C. Voisin, some applications to the Beauville conjecture on the polynomial relations in the Chow group of a projective \hk manifold are given for $J=J(X)$, in the case of very general $J$ of Picard number $2$ or $3$. This is obtained as an application of the computation of $q(\Theta)=-2$ from Theorem \ref{main thm 2}.

\subsection*{Acknowledgements} I would like to thank J. Koll\'ar for pointing my attention to the techniques of \cite{Kollar-Elliptic} and \cite{Lai}, which are used in this paper.  
It is my pleasure to thank E. Arbarello, C. Camere, G. Di Cerbo, K. Hulek, R. Laza, E. Macr\`i, C. Onorati, G. Pearlstein, A. Rapagnetta,  L. Tasin, C. Voisin for useful and interesting discussions related to the topic of this paper. I also thank the anonymous referee for having read this paper very carefully and for many useful comments.
Finally, I would like to warmly thank Coll\`ege de France and \'Ecole Normal Sup\'erieure for the  hospitality and  the great working conditions while the  final version of this manuscript was being prepared. This work is partially supported by NSF Grant DMS-1801818.

\section{A \hk compactification of the intermediate Jacobian fibration for any smooth cubic fourfold}

\label{birational section}

We denote by $X \subset \P^5$ a smooth cubic fourfold, by $\Pfd$  the dual projective space parametrizing hyperplane sections $Y=X \cap H \subset X$, and by $U \subset \Pfd$ the open subset parametrizing smooth hyperplane sections. The dual hypersurface of $X$, parametrizing  singular hyperplane sections, is denoted by $X^\vee \subset \Pfd$.  Its smooth locus 
\[
U_1:=(\P^5)^\vee \setminus Sing(X^\vee) \subset (\P^5)^\vee
\]
parametrizes hyperplane sections of $X$ that are smooth or have one simple node and no other singularities. In what follows, we freely drop the $^\vee$ from $\Pfd$ and write simply $\P^5$. From the context it will be clear if we are referring to the  projective space parametrizing hyperplane sections of $X$ or the projective space containing $X$. For a smooth cubic threefold $Y$, the Griffiths' intermediate Jacobian of $Y$ will be denoted by
\[
\Jac(Y)\cong H^1(Y, \Omega_Y^2)^\vee \slash H_3(Y,\Z).
\]
It is a principally polarized abelian fivefold which parametrizes rational equivalence classes of homologically trivial $1$-cycles on $Y$ \cite[Thm. 6.24]{Voisin-Chow}.

Over $U$ consider the Donagi-Markman fibration
\be \label{intjacfibr}
\pi_U: J_U=J_U(X) \to U
\ee
whose fiber over a smooth hyperplane section $Y=X \cap H$ is the intermediate Jacobian $\Jac(Y)$. By \cite{Donagi-Markman}, $J_U$ is quasi-projective and admits a holomorphic symplectic form $\sigma_{J_U}$ with respect to which $\pi_U$ is Lagrangian.  The main results of \cite{LSV} is the following theorem

\begin{thm}[\hspace{1sp}\cite{LSV}] \label{LSVthm}
Let $X$ be a general cubic fourfold. Then there exists a smooth projective  compactification $J=J(X)$ of $J_U$, with a  flat morphism $\pi: J \to \Pfd$ extending $\pi_U$ which has irreducible fibers and which admits a rational zero section $s: \Pfd \dashrightarrow J$. Moreover, $J$ is an irreducible holomorphic symplectic manifold, deformation equivalent to O'Grady's $10$-dimensional exceptional example.
\end{thm}

We will say that $X$ is general in the sense of LSV if the construction of \cite{LSV} works for $J_U(X)$, and we refer to $J=J(X)$ as in Theorem \ref{LSVthm} as the LSV fibration. A necessary condition for this to happen is that the hyperplane sections of $X$ are palindromic (see \cite{Brosnan}). For example, a cubic fourfold containing a plane is not general in the sense of LSV.

To extend the theorem above for any $X$, we use the existence of a \hk compactifiction  for general $X$, the cycle theoretic description of the holomorphic symplectic form that was given in \cite{LSV}, the degeneration results from \cite{KLSV}, and  techniques from the minimal model program (following \cite{Kollar-Elliptic,Lai}). We start by recalling the construction of a natural partial compactification of $J_U$, which already appeared in \cite{Donagi-Markman, LSV}.

\begin{lemma} [\hspace{1sp}\cite{Donagi-Markman,LSV}] \label{lem Juno}  For any smooth $X$, there is a canonical  partial compactification $J_{U_1}=J_{U_1}(X)$ of $J_U$, with a projective morphism $\pi_{U_1}: J_{U_1} \to U_1$ with irreducible fibers extending $\pi_U$.  $J_{U_1}$ is smooth and has a holomorphic symplectic form $\sigma_{J_{U_1}}$ extending $\sigma_{J_U}$.
\end{lemma}
\begin{proof} This is already proved in \cite[\S 8.5.2 and Thm. 8.18]{Donagi-Markman}. Alternatively, one can use
\cite[Cor. 2.38]{Collino-Murre-I}, and  \cite[Def. 2.2 and 2.9, Prop. 1.4, and Lem. 5.2]{LSV}.
\end{proof}

Before giving an application of the cycle-theoretic construction of the holomorphic symplectic form \cite[\S 1]{LSV}, we  recall the definition of symplectic variety.

\begin{defin} A normal projective variety $M$ is called symplectic if its smooth locus carries a holomorphic symplectic form which extends to a regular (i.e. holomorphic) form on any resolution of singularities of $M$.
\end{defin}

\begin{lemma} \label{lem K eff} \label{cor not uniruled}
Let $\bar J $ be a normal projective compactification of $J_U$. Then
\begin{enumerate}
\item The smooth locus of $\bar J$ admits a homolorphic two form extending $\sigma_{J_U}$. In particular, the canonical class $K_{\bar J}$ of $\bar J$ is effective and is trivial if and only if $\bar J$ is a symplectic variety. 
\item $\bar J$ is not uniruled.
\end{enumerate}
\end{lemma}
\begin{proof}
(1) The first statement is \cite[Thm. 1.2 iii)]{LSV}, while the second follows from the fact that the canonical class of $\bar J$ is the (closure of the) codimension one locus where the generically non-degenerate holomorphic two form is degenerate. (2) Let  $\tilde J \to \bar J$ be a resolution of singularities. By (1), $\tilde J$ has effective canonical class and thus by \cite{Miyaoka-Mori} it is not uniruled.
\end{proof}

The following is an application of the degeneration techniques of \cite{KLSV}. 

\begin{prop} \label{prop hk model}
Let $X$ be a smooth cubic fourfold and let $J_{U}=J_{U}(X)$ be as above. Then there exists a smooth projective \hk manifold $M$  birational to   $J_{U}$ and  of OG$10$-type.
\end{prop}

\begin{proof}
Let $\mc X \to \Delta$ be a family of smooth cubic fourfolds with $\mc X_0= X$. Here $\Delta$ is an open affine  subset of a smooth projective curve, or a small disk. We will use the notation $t=0$ to denote a chosen special point in $\Delta$ and $t \neq 0$ to denote any other point. Up to restricting $\Delta$ if necessary, assume that for $t \neq 0$, $\mc X_t$ is general in the sense of LSV.  By \cite[Prop. 2.10]{LSV}, we can assume that for any $t \neq 0$ all the hyperplane sections of $\mc X_t$ admit a very good line (see  \cite[Def. 2.9]{LSV}). Consider the open set $\mc V = (\P^5)^\vee \times \Delta \setminus Sing(\mc X_0^\vee) \times\{ 0\}$, so that $\mc V_t= (\P^5)^\vee$ for $ t \neq 0$ and  $\mc V_0=U_1 \times \{0\} $ parametrizes the hyperplane sections of  $\mc X_0= X$ that have at most one nodal point and no other singularities. The construction of \cite[\S 5]{LSV} can be carried out in families, yielding a projective morphism
 \[
 \mc J_\mc{V} \to \mc V
 \]
 which is fibered in compactified Prym varieties and is such that, denoting by $\mc J_t$ the fiber of the induced smooth quasi-projective morphism  $ \mc J_\mc{V} \to \Delta$, for $t \neq 0$, $\mc J_t$ is the LSV fibration $J(\mc X_t)$, and $\mc J_0=J_{U_1}(X)$. 
 Let $\widetilde{\mc J} \to \Delta$ be a projective morphism extending $ \mc J_\mc{V} \to \Delta$. The central fiber $\mc J_0$ has a multiplicity one component which contains $J_{U_1}$ as dense open subset. By Lemma \ref{cor not uniruled}, this component is not uniruled. By \cite[Cor. 5.2]{KLSV} there is a birational model $M$ of $J_{U_1}(X)$ that is a \hk manifold, deformation equivalent to the smooth fibers ${\mc J}_t=J(\mc X_t)$, $t \neq 0$.
\end{proof}

 By \cite{Matsushita-Def}, given a \hk manifold $M$ with a Lagrangian fibration $\pi: M \to \P^n$, the locus inside $\Def(M)$ where the Lagrangian fibration deforms is an open subset of the hypersurface where the class $\pi^* \mc O(1)$ stays of type $(1,1)$. However, this fact alone is not enough to imply the existence of a \hk compactification of $J_{U_1}$ for \emph{any} smooth $X$.

This is what we prove in the following Theorem \ref{hkcom}, whose proof uses the mmp following Koll\'ar \cite[\S 8]{Kollar-Elliptic} and Lai \cite{Lai}. In \S \ref{subsectmmp} we will recall some basic facts about the mmp that are needed in the proof of Theorems \ref{hkcom} and \ref{deg lagr 1}. We refer to \cite{Kollar-Mori} and to \cite{Hacon-Kovacs} for the basic definitions and fundamental results.

\begin{thm} \label{hkcom}
For any smooth cubic fourfold $X$, there exists a smooth projective \hk compactification $J=J(X)$ of $J_U(X)$, with a projective flat morphism $\pi: J \to \P^5$ extending $\pi_U$. 
\end{thm}

\begin{proof} 
Let $\bar J \to \P^5$ be any normal projective compactification of $J_{U_1}$ with a regular morphism $\bar \pi : \bar J \to \P^5$. 
By Lemma \ref{lem K eff}, there is a holomorphic two form $\bar \sigma$ on the smooth locus of $\bar J$ extending $\sigma_{J_{U_1}}$, the canonical class $K_{\bar J} \ge 0$ is effective, and $K_{\bar J}=0$ if and only if $\bar J$ is a symplectic variety. 
Since $K_{\bar J}$ is supported on the complement of $J_{U_1}$, $\codim \bar \pi(\Supp(K_{\bar J}) ) \ge 2$. By definition \cite[Def. 7]{Kollar-Elliptic}, this means that $K_{\bar J}$ is  $\bar \pi$-exceptional, if it is non trivial. If this is the case, then by \cite[III 5.1]{Nakayama} (cf. also \cite[Lem 2.10]{Lai}), $K_{\bar J}$ is not $\bar \pi$-nef. More precisely, there is a component of $K_{\bar J}$ that is covered by curves that are contracted by $\bar \pi$ and that intersect $K_{\bar J}$ negatively.

Let $\tilde J \to \P^5$ be a smooth projective compactification of $J_{U_1}$ admitting a regular morphism $\tilde \pi : \tilde J \to \P^5$ and let $K_{\tilde J}$ be  its canonical class. If the effective divisor $K_{\tilde J}$ is not trivial, we use the mmp to contract $\Supp(K_{\tilde J})$ relatively to $\P^5$. Let $H$ be a $\tilde \pi$-ample $\Q$-divisor such that the  pair $(\tilde J, H)$ is klt and $K_{\tilde J}+H$ is relatively big and nef. The mmp with scaling over $\P^5$ (see \S \ref{subsectmmp} below) produces a sequence of birational maps 
\be \label{MMP scaling}
\tilde J=J_0 \stackrel{\psi_0}{\dashrightarrow} J_1  \stackrel{\psi_1}{\dashrightarrow} \cdots \dashrightarrow J_i  \stackrel{\psi_i}{\dashrightarrow} \cdots
\ee
over $\P^5$ (i.e., there are projective morphisms $\pi_i: J \to \P^5$ such that $\pi_0=\wt \pi$ and $\pi_i:=\pi_{i-1} \circ \psi_i^{-1}$) 
and a non increasing sequence of non negative rational numbers  $t_0 =1 \ge t_1 \ge  \dots t_i \ge \dots\ge 0$, with the following properties
\begin{enumerate}
\item For every $i \ge 0$, $K_{J_{i }}+t_i H_i$ is $\pi_i$-big and $\pi_i$-nef.  
\item For every $i \ge 0$, $J_i$ is a $\Q$-factorial terminal compactification of $J_{U_1}$. The fact that the birational morphisms $\psi_i$ are isomorphisms away from $J_{U_1}$  follows from the fact that the $K_{J_i}$-negative rays of the mmp correspond to rational curves that are contained in the support of $K_{J_i}$. Thus, by Lemma \ref{lem K eff} the smooth locus of $J_i$ carries a holomorphic two form $\sigma_i$ extending $\sigma_{ J_{U_1}}$; 
\item $K_{J_{i }}$ is effective and, if not trivial, it has a component covered by $K_{J_{i }} $-negative curves which are contracted by $\pi_i$;
\item The process stops if and only if there exists an $i$ such that  $K_{J_{i}}$ is $\pi_{i}$-nef. This holds if and only if $K_{J_{i}}=0$.
\end{enumerate}

The number of irreducible components of the support of $K_{J_{i}}$ is non increasing, since the birational maps of the mmp extract no divisors. In fact, we claim that this number is eventually strictly decreasing. By (4) above, this happens if and only if the process eventually stops.
 Suppose that this is not the case. Then by Lemma \ref{lemma t goes to zero}, $\lim t_i=0$. Recall, as already observed, that if $K_{J_{k}} \neq 0$, then there exists a component that is covered by $K_{J_k } $-negative curves that are contracted by $\pi_k$. 
 Since we are assuming that $\lim t_i=0$, this implies that  for $i \gg 0$, $t_i$ is small enough so that this component is contained in the relative stable base locus $\mathbb B ((K_{J_k}+ t_i H_k)/\P^5)$. Since by Lemma \ref{lemma sbl}, the divisorial components of $\mathbb B ((K_{J_k}+ t_i H_k)/\P^5)$ are contracted by $J_k\dashrightarrow J_i$, it follows that for $i \gg0$, the number of irreducible components of the effective divisor $K_{J_i} $ is strictly less than the number of components of $K_{J_k} $. Thus, the claim is proved and for some $i \gg 0$, the process gives a model  with $K_{J_{i}}=0$. By Lemma \ref{cor not uniruled}, $\bar J:=J_i$ is a  $\Q$-factorial terminal symplectic  compactification of ${ J_{U_1}}$. Finally, by Proposition \ref{GLR} below $\bar J$ is smooth and the theorem is proved.


\end{proof}

\begin{prop}(Greb-Lehn-Rollenske) \label{GLR}
Let $\bar M$ be a $\Q$-factorial terminal symplectic variety. Suppose that $\bar M$ is birational to a smooth \hk manifold $M$. Then $\bar{ M}$ is smooth. 
\end{prop}
\begin{proof}
This is \cite[6.4]{Greb-Lehn-Rollenske}.
\end{proof}

\begin{rem}
The techniques used to prove the theorem above can be applied to similar contexts to give $\Q$-factorial terminal symplectic compactifications of other quasi-projective Lagrangian fibrations. We plan to come back to this in upcoming work.
\end{rem}

As a consequence of Theorem \ref{deg lagr 1} below, we will give  a slightly stronger version of the theorem  just proved (see Remark \ref{rem other proof}) showing that, given a family of smooth cubic fourfolds whose general fiber is general in the sense of \cite{LSV}, then up to a base change and birational transformations, the corresponding family of LSV intermediate Jacobian fibrations can be filled with a Lagrangian fibered smooth projective \hk compactification of the Donagi-Markman fibration of the limiting cubic fourfold. 

Another approach to Theorem \ref{main thm 1} would be to show that the rational map $M \dashrightarrow \P^5$ induced by the birational map $\phi: M \dashrightarrow J_{U_1}$ of Proposition \ref{prop hk model} is almost holomorphic (see \cite[Def.1]{Matsushita-almost}). By \cite{Matsushita-almost} this would imply the existence of a birational \hk model of $M$ with a regular morphism to $\P^5$. It seems, however, that controlling the mmp of Proposition \ref{prop hk model} to ensure that $M\dashrightarrow \P^5$ is almost holomorphic is not too far from running the relative mmp as in the proof of Theorem \ref{hkcom}.

Given a smooth cubic fourfold $X$, we will refer to both the Donagi-Markman fibration $J_U$ and to any \hk compactification $J$  of $J_U$ as in Theorem \ref{hkcom}, as the intermediate Jacobian fibration. Hopefully, it will be clear from the context which one we are referring to.

\begin{rem} Unlike the compactification of \cite{LSV}, the proof of Theorem \ref{hkcom} is not constructive and, for a given $X$, the   \hk compactification that we show to exists may not be unique.
We will return to this question in Section \ref{section general X}.
\end{rem}

\subsection{The mmp with scaling} \label{subsectmmp}
In this subsection we recall some basic tools and known results from the minimal model program (mmp) that are used to prove Theorems \ref{hkcom} and \ref{deg lagr 1}. For the basic notions and the fundamental results we refer to \cite{Kollar-Mori} and \cite{Hacon-Kovacs}. In this section, by divisor we will mean a $\Q$-divisor.

Let $M$ be a normal $\Q$-factorial variety with a projective morphism $\pi: M \to B$ to a normal quasi-projective variety $B$. Let $\Delta$ be an effective divisor on $M$ and let $H$ be a general divisor on $M$ that is ample (or big) over $B$.
 We assume that the pair $(M, \Delta+H)$ is klt and that $K_M+\Delta+H$ is nef over $B$. 

The mmp with scaling of $H$ \cite[\S 5.E]{Hacon-Kovacs}  produces a sequence of birational maps $\psi_i: M_i \dashrightarrow M_{i+1}$ over $B$, such that:  $M_0=M$, $\Delta_{i+1}=(\psi_i)_*\Delta_{i}$, $H_{i+1}=(\psi_i)_*H{i}$ and $\psi_i$ is the flip or the divisorial contraction for a $(K_{M_i}+\Delta_i)$-negative relative extremal ray $R_i$ over $B$. We let $\pi_{i}$ be the induced regular morphism $M_i \to B$. The sequence is defined inductively in the following way.
Let
\[
t_i=\inf\{ t \ge 0 \, | \, K_{M_i}+\Delta_i+ tH_i \,\, \text{ is nef over } B \}.
\]
If $t_i=0$, then $K_{M_i}+\Delta_i$ is nef over $B$ and the process stops. Otherwise, there is a $0<t' \le t_i$ such that $K_{M_i}+\Delta_i+ t'H_i$ is not nef over $B$. 
By the Cone Theorem  (see \cite[Ch. 3]{Kollar-Mori} or \cite[5.4]{Hacon-Kovacs}), $K_{M_i}+\Delta_i+ t_iH_i$ is nef over $B$ and there exists a $(K_{M_i}+\Delta_i)$-negative extremal ray $R_i$ over $B$ such that $(K_{M_i}+\Delta_i+ t_iH_i) \cdot R_i=0$.

Let $c_i: M_i \to Z_i$ be the extremal contraction over $B$ associated to $R_i$, which exists by the Contraction part of the Cone Theorem  \cite[(5.4.3-4)]{Hacon-Kovacs}).  
If $\dim Z_i < \dim M_i$, then $c_i$ is a Mori fiber space and we stop. 
If $c_i$ is not a Mori fiber space then it is either a divisorial or flipping contraction. 
In the first case, we let  $M_{i+1}=Z_i$ and $\psi_i=c_i$. In second case, we let $\psi_i: M_i \dashrightarrow M_{i+1}$ the $(K_{M_i}+\Delta_i+t'H_i)$-flip (which exists by \cite[Cor. 5.73]{Hacon-Kovacs}). By construction, $\psi_i$ extracts no divisors, meaning that $\psi_i^{-1}$ contracts no divisors.

By the contraction part of the Cone Theorem, the divisor $K_{M_{i+1}}+\Delta_{i+1}+ t_iH_{i+1}$ is nef over $B$. The pair $(M_i,\Delta_{i+1}+ t_iH_{i+1})$ is klt (see \cite[3.42--3.44]{Kollar-Mori}) and $M_i$ is $\Q$-factorial (see \cite[3.18]{Kollar-Mori}). If $\Delta=0$ and $M$ is terminal, then so is $M_i$.  As long as $K_{M_i}+\Delta_i$ is not $\pi_i$-nef, $t_{i+1}$ is non zero and $\Delta_{i+1}+ t_iH_{i+1}$ is big over $B$. Thus we can keep going, producing a non increasing sequence $t_i \ge t_{i+1} \ge \cdots $ of non negative rational numbers and a sequence of birational maps $\psi_i: M_i \dashrightarrow M_{i+1}$ over $B$. The process stops if there exists an $N$ such that $c_N: M_N \to Z_N$ is a Mori fiber space over $B$ or such that $K_{M_N}+\Delta_N$ is nef over $B$. Otherwise, the sequence is infinite. 

 The pair $(M_i,\Delta_i+t_iH_i)$ is a log terminal model (ltm) for $(M, \Delta+t_iH)$ over $B$ (see Definition 5.29 and Lemma 5.31 of \cite{Hacon-Kovacs}).
We will need the following Lemmas:

\begin{lemma} \label{lem not iso}
For any $i>j$, let $\psi_{ij}: M_j \dashrightarrow M_i$ be the induced birational morphism over $B$. Then $\psi_{ij}$ is not an isomorphism.
\end{lemma}
\begin{proof}
This is   \cite[Lem 5.62]{Hacon-Kovacs}.
\end{proof}

\begin{lemma}(\hspace{1sp}\cite[Ex. 5.10]{Hacon-Kovacs})
Let $(M, \Delta)$ be a klt pair as above and suppose that $\Delta$ is big over $B$ and that $K_M+\Delta$ is nef over $B$. Then $K_M+\Delta$ is semiample over $B$, i.e., there exists a projective morphism $f: M \to Z$ over $B$ and an ample divisor $L$ on $B$ such that $K_M+\Delta \sim_{\Q,B} f^*L$.
\end{lemma}
\begin{proof}
Since $\Delta$ is big over $B$, we can write $\Delta\sim_{\Q,B} A+C$, where $A$ is ample over $B$ and $C \ge 0$. Choose an  $0 < \epsilon \ll 1$, such that $(M,\Delta')$ is klt, where $\Delta'=(1-\epsilon)\Delta+\epsilon C$.
Then 
\[
(K_M+\Delta)-(K_M+\Delta')=\epsilon A,
\]
is ample over $B$. By the Basepoint free Theorem  (e.g. see \cite[(5.1)]{Hacon-Kovacs}),  $K_M+\Delta$ is semiample over $B$. 
\end{proof}

\begin{lemma} \label{lemma sbl} Let the notation be as above and for any $i> 0$, let $\phi_i: M \dashrightarrow M_i$ as  the induced birational map over $B$. Then the divisors contracted by $\phi_i$ are the divisorial components of $\mathbb B((K_{M_i}+\Delta_i+ t_iH_i)/B)$, the  stable base locus over $B$ (cf. \cite[\S 2.E]{Hacon-Kovacs}). Similarly, $\psi_{ij}: M_j \dashrightarrow M_i$ contracts the divisorial components of $\mathbb B((K_{M_j}+\Delta_j+ t_iH_j)/B)$.
\end{lemma}
\begin{proof}
Since $(M_i, \Delta_i+ t_iH_i)$ is klt, $\Delta_i+ t_iH_i$ is big over $B$, and $K_{M_i}+\Delta_i+ t_iH_i$ is nef over $B$,  by the lemma above, $K_{M_i}+\Delta_i+ t_iH_i$ is semiample over $B$.

Let $W$ be a smooth birational model resolving $\phi_i$, and let $p$ and $q$ be the induced birational morphisms to $M$ and $M_i$. 
By  \cite[Lemma 5.31]{Hacon-Kovacs} pair $(M_i,\Delta_i+t_iH_i)$ is a log terminal model for $(M, \Delta+t_iH)$ over $B$ (see \cite[Definition 5.29]{Hacon-Kovacs}).
Thus,
\be \label{ltm}
p^*(K_M +\Delta+t_i H)=q^*(K_{M _i}+\Delta_i+t_i H_i)+ E
\ee
where $E=\sum_F (a(F; M,\Delta+t_i H)-a(F; M_i,\Delta_i+t_i H)) F $ is an effective $q$-exceptional divisor whose support contains the divisors contracted by $\phi_i$.  Since
$p^{-1} \mathbb B((K_M+\Delta+t_i H)/B)=\mathbb B(p^*(K_M+\Delta+t_i H)/B)=\mathbb B(q^*(K_{M _i}+\Delta_i+t_i H_i)+ E/B)=\Supp (E)$, the first statement follows. The second statement is proved in the same way, since by  \cite[Lemma 5.31]{Hacon-Kovacs}, the pair $(M_i, \Delta_i+t_i H_i)$ is a log terminal model for $(M_j, \Delta_j+t_i H_j)$ over $B$ and hence the equivalent of (\ref{ltm}) holds.
\end{proof}

\begin{lemma} \label{lemma t goes to zero} Let the notation be as above. If the mmp with scaling does not terminate then
\[
\lim_{i \to \infty} t_i=0.
\]
\end{lemma}
\begin{proof} 
This is  \cite[Prop. 3.2]{Druel}. The only difference is the relative setting, but the proof is the same:
Suppose  the mmp does not terminate and  that $\lim t_i=t_\infty>0$. By \cite[Thm E]{BCHM}  there are finitely many log terminal models of $(M, \Delta+(t_\infty+t) H)$, with $t \in [0,1-t_\infty]$. We have already observed that $(M_i, \Delta_i+t_i H_i)$ is a ltm for $(M, \Delta+t_i H)=(M, \Delta+t_\infty H+(t_i-t_\infty )H))$ over $B$. Thus, if the sequence is infinite there are integers $i >j$ such that the birational map $M_j \dashrightarrow M_i$ is an isomorphism. This gives a contradiction with Lemma \ref{lem not iso} above.
\end{proof}

\subsection{Variants}

In this section we give some variants of the results of the previous section. First we notice that the compactification result of Theorem \ref{hkcom} holds also for the twisted intermediate Jacobian fibration (see Remark \ref{twisted}). Then we consider the case of the intermediate Jacobian fibration associated to a mildly singular cubic  fourfold (see Proposition \ref{checa} and Remark \ref{mild sing}). We then give a slightly stronger version of Theorem \ref{hkcom}, in that we show that the Lagrangian fibered \hk compactification works in families (see Proposition \ref{Csix} and Theorem \ref{deg lagr 1}). As an application, we give a positive answer to a question of Beauville (see Remark \ref{rem beau}).

\begin{rem}[The twisted case] \label{twisted} In \cite{Voisin-twisted}, Voisin constructs a non trivial $J_U$-torsor $J_U^T \to U$ defined from a class in $H^1(U, \mc J_U[3])$, where $\mc J_U$ is the sheaf of holomorphic sections of $J_U \to U$ and where $\mc J_U[3] \subset \mc J_U$ is the sheaf of $3$-torsion points. The non triviality (for very general $X$) of this class corresponds to the non existence, for the universal family of hyperplanes sections of $X$, of a relative $1$-cycle of degree $1$.
The main result of the paper is to produce, for general $X$, a \hk compactification $J^T=J^T(X)$ with Lagrangian fibration to $\P^5$ extending $J_U^T \to U$. This builds on the compactification of \cite{LSV}. 
We will refer to this \hk manifold as the twisted intermediate Jacobian fibration. This \hk manifold is deformation equivalent to the non twisted version $J(X)$, as they agree as soon as $X$ has a $2$-cycle which restricts to a $1$-cycle of degree $1$ or $2$ on its hyperplane sections. 
Lemma \ref{cor not uniruled}, Proposition \ref{prop hk model} and Theorem \ref{hkcom} work the same for the non-trivial torsor $J_U^T \to U$, giving  a Lagrangian fibered \hk $J^T=J^T(X)$ for every smooth $X$. In Section \ref{Kuz} we will return to the twisted intermediate Jacobian fibration and in Corollary \ref{J not JT} we prove that for very general $X$ these two fibrations are not birational and that on $J$ there is a unique isotropic class in the movable cone of $J$. This fact will be used in the Appendix \ref{appendix}.
\end{rem}

Finally, we show that the Lagrangian fibered \hk compactification exists generically also over $\mc C_{6}$, the  divisor in the moduli space of cubic fourfolds whose general point parametrizes cubics with one $A_1$ singularity. The following Proposition is an adaptation of  \cite[\S 2]{LSV} to the case of a cubic fourfold with mild singularities.

\begin{prop} \label{checa}
Let $X_0 \subset \P^5$ be a cubic fourfold with one simple node $o \in X_0$ and no other singularities, let $U \subset \P^5$ be the open locus parametrizing smooth hyperplane sections, and let $\pi_U: J_U=J_U(X_0) \to U$ be the Donagi-Markman fibration. Then, there exits a holomorphic symplectic form $\sigma_U$ on $J_U$, which extends to a holomorphic two form on any smooth projective compactification.  As a consequence, Lemma \ref{cor not uniruled} holds for $J_U$, namely  
 any projective compactification of $J_U$ has smooth locus admitting a generically non-degenerate holomorphic $2$-form extending  $\sigma_U$ and is not uniruled.
Similarly, for the twisted intermediate Jacobian $J_U^T=J_U^T(X_0)$.
\end{prop}
\begin{proof}
Let $\wt X_0$ (resp. $\wt \P^5$) be the blow-up of $X_0$ (resp. $\P^5$) at the point $o$. Let $E \subset \wt \P^5$ be the exceptional divisor. Projection from $o$ determines an isomorphism  $\wt X_0 \cong  BL_S \P^4$, where $S$ is the $(2,3)$ complete intersection in $ \P^3$ parametrizing lines in $X_0$ by $o$. The surface $S$ is a smooth K3 surface and  thus $H^1(\wt X_0, \Omega^3_{\wt X_0} )$ is one dimensional; let $\eta$ be a generator. The same argument as in \cite[Theorem 1.2 ]{LSV} shows that $\eta$ induces a holomorphic two form $\sigma$ on $J_U$, with respect to which the fibers of $J_U \to U$ are isotropic. To show that $\sigma$ is non-degenerate, it suffices to show that for any smooth hyperplane section $Y$ (which in particular does not pass by the point $o$), the map 
\be \label{etat}
T_{[Y]}U=H^0(Y,\mc O_Y(1)) \to H^1(Y, \Omega^2_Y)=H^0(J_U, \Omega^1_{J_U}),
\ee
induced by $\sigma$, via the fact that the fibers of  $J_U \to U$  are isotropic, is an isomorphism. By \cite[Theorem 1.2 (ii)]{LSV}, this map is given by the cup product with a class $\eta_Y \in H^1(Y, \Omega^2_Y(-1))$ defined in the following way:  let $\eta_{|Y} \in H^1(Y, (\Omega^3_{\wt X_0})_{|Y} )$ be the restriction of $\eta$ to $Y$. Since $H^1(Y, \Omega^3_Y)=0$, the exact sequence $0 \to \Omega^2_Y(-1) \to (\Omega^3_{\wt X_0})_{|Y} \to \Omega^3_Y \to 0$ implies that $\eta_{|Y}$ lifts to a class $\eta_Y \in  H^1( \Omega^2_Y(-1))$. By Griffiths residue theory (see \cite[ Lemma 1.7]{LSV}), $H^1( \Omega^2_Y(-1))$ is one dimensional and cup product with any non-zero element induces an isomorphism $H^0(Y,\mc O_Y(1)) \to H^1(Y, \Omega^2_Y)$;
more precisely, using  the canonical isomorphism $\Omega^2_Y(-1)=T_Y(3)$, this space is spanned by the class of the non-trivial extension $0 \to T_Y \to (T_{\P^4})_{|Y} \to \mc O_Y(-3) \to 0$.
It follows that to show that (\ref{etat}) is an isomorphism, we only need to show that $\eta_Y \neq 0$, which amounts to showing that $\eta_{|Y} \neq 0$.
Under the isomorphism $ \Omega^3_{\wt X_0}=T_{\wt X_0}(-3)(2E)$, the class of a generator of $H^1(\wt X_0, \Omega^3_{\wt X_0} )$ corresponds to the class of the extension $ 0 \to T_{\wt X} \to (T_{\wt \P^5})_{|\wt X} \to \mc O_{\wt X}(3)(-2E) \to 0$. Restricting to $Y$ and considering the tangent bundle sequence for $Y$ in $\P^4$ we get the following diagram of short exact sequences
\[
\xymatrix{
0  \ar[r] & (T_{\wt X})_{|Y}  \ar[r] &  (T_{\wt \P^5})_{|Y}  \ar[r] &  \mc O_Y(3)  \ar[r] &  0\\
0  \ar[r] & T_Y  \ar[r] \ar[u] &  (T_{ \P^4})_{|Y}  \ar[u]^\alpha \ar[r] &  \mc O_Y(3) \ar@{=}[u]  \ar[r] & 0
}
\]
where the first two vertical arrows are injective. The extension class of the first row is $\eta_{|Y}$ and the second row is non split, as we already observed.  Since $\coker (\alpha)=\mc O_Y(1)$, then $\Hom(\mc O_Y(3),\coker (\alpha))=0$. Thus any splitting of the first row would induce a splitting of the second row, giving a contraction.
\end{proof}

\begin{rem} \label{mild sing}
The Proposition \ref{checa} holds, more generally, for any cubic fourfold with isolated singularities, as long as a general one parameter smoothing of it has finite monodromy. This corresponds to the K3 surface $S$ of lines through one of the singular points having canonical singularities. The case of the degeneration to the chordal cubic \cite{Hassett-special}, which has finite monodromy but central fiber with $2$--dimensional singular locus, will be discussed at length in Section \ref{deg chordal}.
\end{rem}

\begin{prop} \label{Csix}
Let $X_0 \subset \P^5$ be as in Proposition \ref{checa} (or as in Remark \ref{mild sing}) and let $\pi_U: J_U \to U$  be the corresponding intermediate Jacobian fibration. Then there exists a \hk compactification $J=J(X_0)$ of $J_U$, with a regular flat morphism to $\Pfd$ extending $\pi_U$. Moreover, if $\mc X \to \Delta$ is a general family of smooth cubic fourfolds degenerating to $X_0$, then up to a base change, there exists a family of Lagrangian fibered \hk manifolds 
\[
\mc J \to \P^5_\Delta \to \Delta
\]
such that for $t \neq 0$, $\mc J_t=J(\mc X_t)$ is the LSV compactification and, for $t=0$, $\mc J_0$ is a \hk compactification of $J_U=J_U(X_0)$. Similarly, the analogue statement holds for the twisted intermediate Jacobian.
\end{prop}

\begin{proof}
By Proposition \ref{checa} above, $J_U$ has a holomorphic symplectic form that extends to a regular form on any smooth projective compactification. As in Lemma \ref{cor not uniruled}, it follows that $J_U$ is not uniruled. Let $\mc X \to \Delta$ be a  family of smooth cubic fourfolds degenerating to $\mc X_0=X_0$ with the property that for $t \neq 0$ $\mc X_t$ is general in the sense that of LSV. As in the beginning of Proposition \ref{hkcom}, let  $\mc J_{\mc V} \to \mc V$ be such that the fiber over $t \neq 0$ of $\mc J_{\mc V} \to \Delta$ is the LSV compactification $J(\mc X_t)$ and, over $t=0$, is $J_U \to U$.
We are thus in the position of applying Theorem \ref{deg lagr 1} below, which proves the proposition.
\end{proof}

A consequence of this proposition is a positive answer to a question of Beauville  \cite{Beauville-Fano-K3}, as explained in the following remark.

\begin{rem} \label{rem beau}
Given a smooth cubic threefold $Y$, let $\ell \subset Y$ be a line. In \cite{Beauville-cubicthreefolds, Druel}, it is shown that the moduli space of Ulrich bundles on $Y$ with rank $2$, $c_1=0$ and $c_2=2\ell$ is birational to the intermediate Jacobian of $Y$ (more precisely, it can be identified with the blowup of the intermediate Jacobian fibration along the Fano surface).
Now let $X_0$ be cubic fourfold with one simple node and let $S \subset \P^4$ be the $(2,3)$ complete intersection K3 surface  parametrizing lines through the singular point of $X_0$. Consider the Mukai vector $v=2v_0=2(1,0,-1)  \in H^*(S,\Z)$ and let $\wt M_{2v_0}(S)$ be the symplectic resolution of the singular moduli space of OG$10$-type (cf. \S \ref{ogtentype}). \\
By considering the relative moduli spaces of Ulrich bundles supported on the $5$-dimensional family of cubic threefolds containing $S$ and by restricting the bundles to $S$, Beauville \cite[\S 5, Example $d=3$]{Beauville-Fano-K3} shows that there is a birational map $ J_U \dashrightarrow M_v(S)$. This induces a rational map $M_{2v_0}(S)\dashrightarrow \P^5$ and Beauville asks whether there exists a \hk manifold birational to $M_v(S)$ which admits a \emph{regular} morphism to $\P^5$.
Proposition \ref{Csix} thus gives a positive answer to this question.
\end{rem}

The proof of the proposition above relies on the following Theorem, which is the Lagrangian fibration analogue of results from \cite[Thm 2.1 and Cor. 5.2]{KLSV}. Theorem \ref{deg lagr 1} will be used also in Section \ref{section general X} for the proof of Proposition \ref{3fams} (and thus also of Theorem \ref{NSU}).
As usual, $\Delta$ is an open affine subset of a smooth curve, or a small analytic disk. In both cases, we keep the notation $t=0$ to denote a chosen special point in $\Delta$ and $t \neq 0$ to denote any other point. 

\begin{thm} \label{deg lagr 1} Let $\tilde f: \wt{\mc J} \to \Delta$ be a projective degeneration of  \hk manifolds of dimension $2n$. Suppose that there is a commutative diagram
\[
\xymatrix{
\wt{\mc J} \ar[dr]_{\tilde f} \ar[r]^{\tilde \pi} & \P^n_\Delta \ar[d]^p \\
& \Delta
}
\]
where $\wt{\mc J} \to \P^n_\Delta$ is a projective fibration such that for  $t \neq 0$, $\mc J_t \to \P^n_t$ is a Lagrangian fibration. Assume that the central fiber $\wt{\mc J}_0= Y_0 +\sum_{i \in I} m_i Y_i$ has a reduced component $Y_0$ which is not uniruled. Suppose, furthermore, that there is an open subset of $ Y_0  \setminus \cup _{ i \ge 1} (Y_i \cap Y_0)$ such that the morphism to $\P^n_0$ is a fibration $J_{U_0} \to U_0 \subset \P^n_0$ in abelian varieties. Then
\begin{enumerate}
\item There exists a projective degeneration $\bar f: \Jbmc \to \Delta$ of \hk manifolds such that:  
\begin{enumerate}
\item $\Jbmc$ is  $\Q$-factorial, terminal, and  isomorphic to $\wt{\mc J}$ over $\Delta^*$;
\item  The central fiber $\Jbmc_0$ is a reduced, irreducible, and a normal symplectic variety with canonical singularities and admitting a symplectic resolution;
\item There is a relative Lagrangian fibration $\bar \pi: \Jbmc \to  \P^n_\Delta$ compatible, via the birational map $ \Jbmc \dashrightarrow  \wt{\mc J}$, with  $\tilde \pi$ and such that, up to restricting the open set $ U_0 \subset \P^n_0$,  the morphism $\Jbmc_0 \to \P^n_0$ extends the abelian fibration $J_{U_0} \to U_0$.
\end{enumerate}

\item Up to a base change $\Delta' \to \Delta$, there exists a (non necessarily projective) family $\mc J \to {\Delta'}$ of \hk manifolds,  with a birational morphism $\mc J  \to \Jbmc':= \Jbmc \times_{\Delta'} \Delta$ over $\Delta'$, which is an isomorphism away from the central fiber and in the central fiber is a symplectic resolution of $\Jbmc_0$. Moreover, $\mc J$ has a family of Lagrangian fibrations $ \pi{'}: \mc J   \to  \P^n_{\Delta'}$ compatible with the base change of $\bar \pi$.
\end{enumerate}
\end{thm}

\begin{proof}
The proof follows ideas from \cite{Takayama,Kollar-Elliptic,KLSV}. Up to passing to a log resolution of the pair $(\wt{\mc J}, \wt{\mc J}_0)$, we can assume that $\wt{\mc J}_0= Y_0 +\sum_{i = 1}^k m_i Y_i$ is a normal crossing divisor. By \cite[Thm 2.1 and Cor 5.2]{KLSV}, running the mmp over $\Delta$ contracts the components $Y_i$, for $i \ge 1$, and yields a birational model of $\tilde{\mc J}$ with an irreducible central fiber which is a symplectic variety. In particular, $Y_0$ is the unique component of  $\wt{\mc J}_0$ that is not uniruled (cf. \cite[Remark 2.2]{KLSV}). To prove the theorem we only need to show that the birational maps required to contract the other components can be preformed relatively to $\P^n_\Delta$ and, furthermore, that they induce isomorphism away from $\cup _{ i \ge 1} Y_i$. This is to ensure that the central fiber has a Lagrangian fibration extending $J_{U_0} \to U_0$ (maybe up to restricting the open subset $U_0 \subset \P^n_0$).

The canonical class $K_{\wt{\mc J}}$ is trivial over $\Delta^*$, so it is $\tilde f$-equivalent to divisor of the form $\sum_{i = 0}^k a'_i Y_i$. Following \cite[2.3 (1)]{Takayama} we set $r=\min{a'_i/m_i}$, so\footnote{For a projective morphism $f:A \to B$ and two $\Q$-Cartier divisors $D$ and $D'$ on $A$, we write $D=_{\Q, B}D'$ or $D\sim_{\Q, f}D'$  iff $D$ and $D'$ are $\Q$-linearly equivalent up to the pullback of a $\Q$-Cartier divisor from $B$.}
\[
K_{\wt{\mc J}}=_{\Q, \Delta} \sum_{i = 0}^k a_i Y_i,
\]
where $a_i=a'_i-rm_i \ge 0$ are non negative rational numbers and for at least one $i$, $a_i=0$. Let $J \subsetneq \{0, 1, \dots, k\}$ be the set of indices such that $a_i >0$ and let $J^c$ be its complement.
By \cite[Prop. 5.1]{Takayama}
\begin{enumerate}
\item For every $j \in J$, the irreducible component $Y_j$ is uniruled.
\item If $|J^c| \ge 2$, then for every $j \in J^c$, the irreducible component $Y_j$ is uniruled.
\end{enumerate}

Since $Y_0$ is not uniruled, it follows that $J=\{1, \cdots, k\}$ and thus
\[
K_{\wt{\mc J}}=_{\Q, \tilde \pi} \sum_{i=1}^k a_i Y_i, \quad \,\, a_i > 0.
\]

By assumption, for every $i \ge 1$, the closed subset $Y_0 \cap Y_i$ is in the complement of $J_{U_0}$ and, since the fibers of $\wt{\mc J}_0 \to \P^n_\Delta$ are connected, it follows that the induced map $Y_i \to  \P^n_0$ is not dominant. Thus,  the codimension of $ \tilde \pi(Y_i) $ in $\P^n_\Delta$ is greater or equal to $2$. In other words, $Y_i$ is $\tilde \pi$-exceptional.

We are in the same setting of Theorem \ref{hkcom}, namely a projective morphism from a smooth quasi-projective variety with a canonical class that is relatively $\Q$-linearly equivalent to an effective divisor all of whose components are relatively exceptional.
We can thus argue as in the proof of Theorem \ref{hkcom}, running the mmp over $\P^n_\Delta$ with scaling of an ample divisor in order to contract each of the $Y_i$, $i \ge 1$. This yields a birational map $\wt{\mc J} \dashrightarrow \Jbmc$ over $\P^n_\Delta$, where  $\Jbmc \to \Delta$ has irreducible fibers and the fibration $\Jbmc \to  \P^n_\Delta$ has $\Q$-factorial terminal total space and is such that $K_\Jbmc= \tilde \pi^*  B $, for some $\Q$-divisor $B$ on $\P^n_\Delta$.  Since  at each step the $K$-negative rays of the mmp are contained in uniruled components of the central fiber, it follows that the birational map $\wt{\mc J} \dashrightarrow \Jbmc$ is an isomorphism away from $\cup_{i \ge 1} Y_i$. In particular, the central fiber $\Jbmc_0$, which is irreducible, has an open subset which is isomorphic to $ J_{U_0} $.
Since for $t\neq 0 $, $(K_{\Jbmc})_{|\mc J_t}=0$,  $B_{|\mc \P^n_t}=0$ for $t\neq 0$. In particular, $B$ is $p$-trivial, where $p: \P^n_\Delta \to \Delta$ is the projection, and thus $K_{\Jbmc}$ is $\tilde f$-trivial. We can now argue as in the last part of the proof of  \cite[Theorem 1.1]{LSV} to show that  $\Jbmc_0$ is a normal with canonical singularities. As in \cite[Corollary 4.2]{LSV} it follows that $\Jbmc_0$ is a symplectic variety and that, up to a base change $\Delta' \to \Delta$ there  exits a smooth family $\mc J \to \Delta'$ with a birational morphism $\mc J  \to \Jbmc':= \Jbmc \times_{\Delta'} \Delta$ with the desired properties.
\end{proof}

\begin{rem} \label{rem other proof}
 Theorem \ref{deg lagr 1} gives another proof of Theorem \ref{hkcom}, as well as the stronger statement of the existence of a relative intermediate Jacobian fibration $\mc J \to \P^5_\Delta$, associated to any family $\mc X \to \Delta$ of smooth cubic fourfolds for which the general fiber is general in the sense of LSV.
\end{rem}


\section{Moduli spaces of OG$10$-type} \label{ogtentype}

By \cite[Cor 6.3]{LSV} (see also \cite[\S 6.3]{KLSV}) any \hk compactification $J$ of $J_U$ is deformation equivalent to O'Grady's $10$-dimensional example. We start this section by recalling the basic definitions and first properties of those singular moduli spaces of sheaves on a K3 surface whose symplectic resolutions are \hk manifolds in this deformation class. Then we use the methods of Bayer-Macr\`i, as adapted by Meachan-Zhang to this class of singular moduli spaces, to study the movable cone of certain moduli spaces that appear naturally as limits of the intermediate Jacobian fibration, when the underlying cubic fourfold degenerates to the chordal cubic (cf. Section \ref{deg chordal}).

We start by recalling the following fundamental theorem.

\begin{thm}[\hspace{1sp}{\cite{Mukai,Yoshioka,OGrady99,Lehn-Sorger, Kaledin-Lehn-Sorger,Perego-Rapagnetta}}] \label{OG10} Let $(S,H)$ be a general polarized K3 surface and let $v_0 \in H^*_{alg}(S,\Z)$ be a primitive Mukai vector which we suppose to be positive in the sense of  \cite[Def. 5.1]{Bayer-Macri-proj} (cf. also \cite[Rem. 3.1.1]{dCRS}). Let $m \ge 2$ be an integer. The moduli space $M_{mv_0, H}(S)$ of $H$--semistable sheaves on $S$ with Mukai vector $mv_0$  is an irreducible normal projective symplectic variety of dimension $m^2v_0^2+2$, which admits a symplectic resolution if and only if $m=2$ and $v_0^2=2$. 
When this is the case,  the symplectic resolution $\wt M_{2v_0, H}(S) \to M_{2v_0, H}(S)$ is the blow up of the singular locus $ \Sym^2 M_{v_0, H}(S) \subset M_{2v_0, H}(S)$, with its reduced induced structure. Moreover,  $\wt M_{2v_0, H}(S) $ is an irreducible holomorphic symplectic manifold and its deformation class is independent of $(S, H)$ and of $v_0$; in particular, $\wt M_{2v_0, H}(S) $ is deformation equivalent to O'Grady's  original $10$-dimensional exceptional example. 
\end{thm}

We will refer to a Mukai vector of the form $2 v_0$ with  $v_0^2=2$ as a Mukai vector of OG$10$--type and to a \hk manifold in this deformation class as a \hk of OG$10$-type.

\subsection{Contracting the relative theta divisor on relative Jacobian of curves} \label{ogten}

It is known \cite{Arcara-Bertram, Bayer-Macri-MMP, Bayer-Macri-proj, Bayer-BN}  that the birational geometry of moduli spaces of pure dimension one sheaves on a K3 surface is related to Brill-Noether loci. For example, on the degree $g-1$ Beauville-Mukai system of a genus $g$ linear system on a K3 surface, the relative theta divisor can be contracted, possibly after performing a finite sequence of birational transformations. This is the content of the following example.

\begin{example} [\hspace{1sp}\cite{Arcara-Bertram,Bayer-BN}]
Let $(S,C)$ be a general  polarized K3 surface of genus $g$, with $\NS(S)=\Z C$. Set $v=(0,C,0) \in H^*(S, \Z)$ and let $M_{v}$ be the moduli space of $C$-stable sheaves on $S$ with Mukai vector\footnote{This Mukai vector is not positive in the sense defined above, since both the first and last entry are zero. However, since for general $(S,C)$, tensoring by $C$ induces an isomorphism with $M_{v'}$, where $v'=(0,C,g-1)$,  the results of \cite{Bayer-Macri-MMP} still hold. See also \cite{Perego-Rapagnetta} for other considerations about the last entry of the Mukai vector.} $v$. Since we are assuming $(S,C)$ to be general in moduli, we are suppressing the polarization from the notation (thus $M_v$ will denote the moduli space of $C$-semistable sheaves on $S$ with Mukai vector $v$; when we consider instead a Bridgeland stability condition $\sigma$, the corresponding moduli space will be denoted $M_{v, \sigma}$). This moduli space is smooth and $M_{v} \to \P^g=|C|$ is the degree $g-1$ relative compactified Jacobian of the genus $g$ linear system $|C|$ on $S$. There is a naturally defined effective, irreducible, relatively ample theta divisor $\theta \subset M_{v}$ which parametrizes sheaves with a non trivial global section and which can be realized  as the zero locus of a canonical section of the determinant line bundle (see \cite[\S 2.3]{LePotier} or  \cite[Thm 5.3]{Alexeev-compactified}). Recall that there is a Hodge isometry $\NS(M_{v})\cong v^\perp= \langle (0,0,1), (1,0,0)\rangle$ (see for example \cite[Theorem 3.6 ]{Bayer-Macri-MMP}).

The class $\ell:=(0,0,1)$ is the class of the isotropic line bundle inducing the Lagrangian fibration  $M_{v} \to \P^g$ while the theta divisor $\theta$ corresponds to the class $-(1,0,1)=-v(\mc O_S)$ (see \cite[pg. 643]{LePotier} or also \cite[Prop. 7.1 and Thm 12.3]{Bayer-Macri-MMP}\footnote{Compared to \cite{Bayer-Macri-MMP}, there is a difference in a choice of sign in the isomorphism $\NS(M_{v})\cong v^\perp$.}).  Since $\theta^2=-2$, the irreducible effective divisor $\theta$ is prime exceptional. By \cite{Druel-exc}, it can be contracted on a \hk birational model of $M_v$. Since the rays corresponding to divisorial contractions and to Lagrangian fibrations must be in the boundary of the movable cone \cite{Huybrechts-kahler-cone}, it follows that
\[
\overline{ \Mov}(M_{v})=\R_{\ge 0}\ell +\R_{\ge 0}h
\]
where $h=(-1,0,1) \in \theta^\perp \cap v^\perp$ is a big line which is nef on some birational model of $M_{v}$ (this also follows from \cite[Thm 12.3]{Bayer-Macri-MMP}). Using \cite{Bayer-Macri-MMP}, the walls of the nef cones of the various birational models can be computed. Since we don't need this, we omit the computation.
\end{example}

\subsection{Movable cones of certain moduli spaces of OG$10$-type.}  \label{movog10} If we consider a non primitive genus $g$ linear system $|mC|$, $m \ge 2$, then the relative compactified Jacobian of degree $g-1$ is singular. For singular moduli spaces of OG$10$-type, i.e. when $v=2w$ with $w^2=2$, \cite{Meachan-Zhang} have adapted the techniques of Bayer-Macr\`i \cite{Bayer-Macri-proj,Bayer-Macri-MMP} to compute the nef and movable cones of these moduli spaces. We  refer to \cite{Bridgeland,Arcara-Bertram,Bayer-Macri-proj,Bayer-Macri-MMP} for the relevant definitions and main results on Bridgeland stability conditions on K3 surfaces and to \cite{Meachan-Zhang} for the results on moduli spaces of OG$10$ type.

By \cite[Thm. 7.6 (3)]{Meachan-Zhang}, all birational models of $M_{2w}=M_{2w, C}$ which are isomorphic to $M_{2w}$ in codimension one are isomorphic to a Bridgeland moduli space $M_{2w, \sigma}$, for some Bridgeland stability condition  $\sigma$ on $S$. Moreover, by \cite[Cor. 2.8]{Meachan-Zhang}
\be \label{NSsing}
\NS(M_{2w, \sigma}) \cong w^\perp.
\ee

 We now apply the results  of \cite{Meachan-Zhang} to describe the nef and movable cones of certain singular models of OG$10$ appearing as limits of the intermediate Jacobian fibration. By \cite{Perego-Rapagnetta-factoriality}, the factoriality properties of a singular moduli space $M_{2w}$ of OG$10$ type depend on the divisibility of the primitive Mukai vector $w \in H^*_{alg}(S,\Z)$.  More precisely, by \cite[Thm 1.1]{Perego-Rapagnetta-factoriality}, $M_{2w}$ is factorial if and only if $w \cdot u \in 2\Z$ for every $u \in H^*_{alg}(S, \Z)$. Otherwise, $M_{2w}$ is $2$-factorial. Since there can be different birational models with different factoriality properties (cf. Remark \ref{g12}), it is important to choose the correct model to work with.

Now let $(S,C)$ be a general K3 surface of degree $2$ and set
\be \label{vkappa}
v_k:=(0, C, k-2),
\ee
The Le Potier morphism $\pi: M_{2v_k} \to \P^5$ realizes the singular moduli space  $M_{2v_k}$ as a compactification of the degree $2k$ relative Jacobian of  the genus $5$ hyperelliptic linear system $|2C|$.  Composing $\pi$ with the sympectic resolution $m: \wt M_{2v_k} \to M_{2v_k} $, we get a natural Lagrangian fibration:
\be \label{lagrfibr}
\wt \pi: \wt M_{2v_k} \to \P^5
\ee
By the result of Perego-Rapagnetta mentioned above, $M_{2v_k}$ is factorial if and only if $k$ is even. It turns out that the birational class of these moduli spaces is independent of $k$, but the isomorphism class depends on the parity of $k$ (see  \cite[Proposition 3.2.7]{dCRS}): Indeed, tensoring a pure dimension one sheaf by $\mc O_S(C)$ determines an isomorphism 
\be \label{isokeven}
M_{2v_k} \stackrel{\sim}{\to} M_{2v_{k+2}}.
\ee

\begin{rem} \label{g12}
Tensoring a line bundle supported on a smooth hyperelliptic curve of genus $5$ by the unique $g^1_2$ on the curve defines a birational morphism $M_{2v_k} \dashrightarrow M_{2v_{k+1}}$
(I thank A. Rapagnetta for pointing out this to me). As a side remark, notice that the map thus defined is \emph{not} an isomorphism in codimension one. Indeed, it can be checked that when passing to the birational morphism $\wt M_{2v_k} \dashrightarrow \wt M_{2v_{k+1}} $ between the two resolutions, which is an isomorphism in codimension $2$, the exceptional divisor of one model is exchanged with the proper transform of the locus parametrizing sheaves on reducible curves on the other model.
\end{rem}

In view of Lemma \ref{collino} below and the isomorphism (\ref{isokeven}), we will focus on the case $k=0$. 

\begin{rem} \label{zerosect}
For general $(S, C)$ it is not hard to check that the structure sheaf of every curve in $|2C|$ satisfies the numerical criterion for $C$-stability and hence that the fibration $M_{2v_0} \to \P^5$ admits a regular zero section. Notice also that the image of this section is not contained in the singular locus of $M_{2v_0}$.
\end{rem}

By \cite[Cor. 2.8]{Meachan-Zhang},  $\NS(M_{2v_0}) \cong v_0^\perp= U= \langle (0,0,1), (-1,C,0)\rangle$, where, as above,
\be \label{ell}
\ell=(0,0,1)
\ee
is the line bundle inducing the Lagrangian fibration $\pi: M_{2v_0} \to \P^5=|2C|$.  Under the isomorphism $M_{2v_2} \cong M_{2v_{0}}$, induced by tensoring with $\mc O_S(-C)$, the relative theta divisor is mapped isomorphically to the prime exceptional divisor
\be \label{deftheta}
\theta:=-(1,- C, 2),
\ee
parametrizing sheaves which receive a non trivial morphism from the spherical object $\mc O_S(-C)$ (see also Lemma \ref{MZ}). Indeed, the relative theta divisor in $M_{2v_2}$ parametrizes sheaves with a non trivial morphism from $\mc O_S$ and thus its image in $M_{2v_0}$ is exactly the divisor $\theta$.
Notice that
\be \label{thetasq}
\theta^2=-2.
\ee
For later use we highlight the following remark:

\begin{rem} \label{Alexeev}
The effective divisor $\theta \subset M_{2v_0}$ with cohomology class (\ref{deftheta}) does not contain the singular locus of $M_{2v_0}$: Using the description of $\theta$ as the zero locus of a section of the determinant line bundle  \cite[Thm 5.3]{Alexeev-compactified}, which is compatible with $S$-equivalence classes, it is enough to show that the section defining $\theta$ is not identically zero on the singular locus of $M_{2v_0}$. It is therefore sufficient to show that there are $S$-equivalence classes of polystable sheaves all of whose members have a zero space of global sections. This is clear, since the generic semistable sheaf with Mukai vector $2v_0$ is an extension of two degree $1$ line bundles each supported on two distinct curves of genus $2$. 
\end{rem}

The following Lemma is an application of \cite[Thms 5.1-5.2-5.3]{Meachan-Zhang} to $M_{2v_0}$ (N.B. Example 8.6 of loc. cit. is for odd $k$, so in view of Remark \ref{g12} it is concerned with a birational model of $M_{2v_0}$ which is not isomorphic in codimension one and hence we cannot immediately apply it here).

\begin{lemma} \label{MZ}
Let the notation be as above. Then
\[
\Nef(M_{2v_0})=  \R_{\ge 0}{\ell}+\R_{\ge 0}h_0 , \quad \Mov(M_{2v_0,C})=  \R_{\ge 0} \ell+\R_{\ge 0} h .
\]
where
\[
\ell = (0,0,1), \quad h_0=(-1,C,1), \quad h=(-1,C,0).
\]
Moreover, the wall spanned by $h_0=(-1,C,1) $ contracts the zero section of $M_{2v_0} \to \P^5$ and the class corresponding to $h=(-1,C,0)$ is big and nef on the Mukai flop of $M_{2v_0} $ along the zero section and contracts the proper transform of $\theta$.

\end{lemma} 

\begin{proof} Since $\ell=\pi^* \mc O_{\P^5}(1)$ is nef and isotropic, it is one of the two rays of both the Nef and the Movable cones of $M_{2v_0}$. By \cite[Thm 5.3]{Meachan-Zhang} there is a divisorial contraction of BNU-type (notation as in loc. cit), determined by the spherical class $s=(1,-C,2)$, which is orthogonal to $v_0$. 
The second ray of the movable cone is thus determined by $s^\perp \cap v_0^\perp$. We pick $h=(-1,C,0)$ as a generator of this ray, since $h \cdot \ell >0$.
By the same theorem in \cite{Meachan-Zhang}, the flopping walls are determined by $w^\perp \cap v_0^\perp$ for $w$ spherical and such that $w \cdot v_0=2$. There is a unique ray in $\Mov( M_{2v_0})$, that is of this form. It is determined by $w=(1,0,1)=v(\mc O_S)$ or, equivalently, by $w'=(-1,2C, -5)=2v_0-w$.  We can choose $h_0=(-1,C,1)$ as generator of this ray. 
As in \cite[Remark 8.5]{Meachan-Zhang}, we can see that this wall corresponds to the flop of the $\P^5$ corresponding to the sheaves with a morphism from $\mc O_S$, i.e., of the image of the zero section.

\end{proof}

\begin{rem}
It can be shown that the birational model on the other side of the wall can be identified with the Gieseker moduli space $M_{2w_0}$, where $w_0=(2,C,0)$. 
Since we don't need this in the rest of the paper, we omit the proof.
\end{rem}

\begin{rem} \label{thetaamp}
The theta divisor $\theta$ is Cartier, since by \cite{Perego-Rapagnetta} $M_{2v_0}$ is factorial (see also \S \ref{movog10}). Moreover, it is relatively ample over $\P^5$, since by the description of the Nef cone of Lemma \ref{MZ} we can write  $\theta$ as a sum of an ample line bundle and a multiple of $\ell=\pi^* \mc O_{\P^5}(1)$. 
\end{rem}

\section{The relative theta divisor on the intermediate Jacobian fibration.} \label{section theta}

For any smooth cubic threefold $Y$, there is a canonically defined theta divisor in $\Jac(Y)$, which is $(-1)$-invariant and whose unique singular point lies at the origin. For the \hk compactification $J=J(X) \to \Pfd$  of the intermediate Jacobian fibration associated to a smooth cubic  $4$-fold $X$, there is an effective relative theta divisor $\Theta \subset J$, which is defined as the closure of the union of the canonical theta divisor in the smooth fibers. More precisely, by 
 \cite{Clemens-Griffiths,Clemens} (see also \cite[Lem. 5.4]{LSV}), $\Theta $ can be defined
 as the closure of the image of the Abel-Jacobi difference mapping
\be \label{defTheta}
\begin{aligned}
\mc F \times_{\Pfd} \mc F & \dashrightarrow J\\
 (\ell, \ell', Y) & \longmapsto  \phi_Y(\ell-\ell')
\end{aligned}
\ee
 The relative theta divisor $\Theta$ played an important role in \cite{LSV}, where it was shown that for general $X$ the divisor $\Theta$ is $\pi$-ample and $J$ is identified with the relative Proj of the sheaf of $\mc O_{\P^5}$-algebras associated with this divisor. Another useful way of realizing the Theta divisor is using twisted cubics \cite{Clemens}.
Let $Z=Z(X)$ be the Lehn-Lehn-Sorger-van Straten $8$fold \cite{LLSvS}. Then $Z$ is the blowdown $g: Z' \to Z$ of a smooth $10$ fold $Z'$ whose points parametrize nets of (generalized) twisted cubics. The exceptional locus of $g$ parametrizes non ACM cubics and its image in $Z$ is isomorphic to the cubic itself.
Let
\[
r: \P_{Z'} \to {Z'}
\]
 be the $\P^1$-bundle over $Z'$ whose fiber over a twisted cubic $[C] \in Z'$ is the pencil $\P^1_{ C }$  of hyperplane sections of $X$ containing $\Sigma_C:=X \cap \langle C \rangle$. Here $ \langle C \rangle=\P^3$ is the linear span of the curve. By \cite[Sublemma 5.5]{LSV} (see also \cite[\S 4]{Clemens} or \cite[Prop. 6.10]{KLSV}), the Abel-Jacobi map
\be \label{AJtwisted}
\begin{aligned}
\varphi: \P_{Z'} & \dashrightarrow J\\
(C, Y) & \longmapsto  \phi_Y(C-h^2)
\end{aligned}
\ee
 is birational onto its image, which is precisely $\Theta$. Here $h^2$ is the class of the intersection of two hyperplanes in $Y$.

 \begin{rem} \label{nonACM}
 For later use, we remark the following two facts. First of all that the restriction of $\P_{Z'}$ to the locus of non-CM cubics is mapped to the zero section of $J \to \P^5$ (which lies in $\Theta$).
Second, using the Gauss map (see \cite[\S 12]{Clemens-Griffiths} or also \cite[\S 3]{Harris-Roth-Starr}), one can see that if $C $ is a twisted cubic in a smooth cubic threefold $Y$ with the property that $ \phi_Y(C-h^2)=0$ in $\Jac(Y)$, then the cubic surface $\Sigma_C= Y \cap \langle C \rangle$ is singular. 
\end{rem}

For every $X$, the N\'eron-Severi group of $J=J(X)$ has at least rank $2$, since
\[
\NS(J(X)) \supset  \langle L, \Theta \rangle.
\]
Here $L=\pi^* \mc O_{\P^5}(1)$ and $\Theta$ is, as above, the relative theta divisor obtained as the closure of the image of (\ref{defTheta}).

\begin{lemma} \label{tr lattice} For any smooth $X$, there is an isomorphism of rational Hodge structures
$H^2(J, \Q)_{tr} \cong H^4(X, \Q)_{tr}$. In particular, $\rho(J)=\rk H^{2,2}(X,\Q)+1$.
\end{lemma}
\begin{proof}
The first statement was already noted in \cite{LSV}, while the second  follows from the first and the fact that $b_2(J)=24$ and $b_4(X)=23$. 
\end{proof}

\begin{rem} \label{section deform}
The locus, inside  $\Def(J)$, parametrizing intermediate Jacobian fibrations is of codimension $2$ and corresponds to the locus where the classes $L$ and $\Theta$ stay of type $(1,1)$. By \cite[Thm 6]{Sawon}, a Lagrangian fibration with a section deforms, as Lagrangian fibration with a section, over a smooth codimension $2$ locus of the  deformation space of the underlying \hk manifold. Since by Theorem \ref{LSVthm} for general $X$ the LSV compactification $J(X)$ has a section, it follows that the codimension $2$ locus where $L$ and $\Theta$ stay algebraic is exactly the locus where the section deforms.
\end{rem}

We highlight the following corollary for future reference. 

\begin{cor} \label{onlyJ} For very general $X$, $\rho(J)=2$. Thus,
\begin{enumerate}
\item $J$ is the only projective \hk birational model of $J_U$ where $L$ is nef. In particular, any \hk compatification of $J_U$  with a Lagrangian fibration extending $J_U \to U$ is isomorphic to the \cite{LSV} compactification.
\item There is at most one prime exceptional divisor on $J$.
\end{enumerate}
\end{cor}
\begin{proof}
(1) Since $\rho(J)=2$,  the boundary of movable cone of $J$ has two rays, of which $L$ is  one. (2)  If there is a prime exceptional divisor, its class has to be orthogonal to the second extremal ray of the movable cone \cite[Thm 1.5]{Markman-Survey}. Since two prime exceptional divisors with proportional classes have to be isomorphic \cite[Cor. 3.6 (3)]{Markman-prime-exceptional}, there is at most one prime exceptional divisor. 
\end{proof}

The following Lemma was communicated to me by K. Hulek and R. Laza. I thank them for sharing this observation with me and for raising the question of computing $q(\Theta)$.

\begin{lemma} \label{LTheta}
$q(L, \Theta)=1$. In particular, $\langle L, \Theta \rangle$ is a primitive sublattice of $\NS(J)$, isomorphic to the standard hyperbolic lattice $U$ of rank $2$. For very general $X$, $\NS(J)=U$.

\end{lemma}
\begin{proof} The computation of $q(L, \Theta)$ goes as in \cite[Lemma 1]{Sawon-discr}: one expands in $t$ the Fujiki equality $q(L+t \Theta)^5=c(L+t \Theta)^{10}$, where $c=945$ is the Fujiki constant  \cite{Rapagnetta} and uses the fact that $\Theta^5L^5=(\Theta_{|J_{[H]}})^5=5!$. The final statement follows from Lemma \ref{tr lattice}.
\end{proof}

I thank C. Onorati for many discussions around $\Theta$ and for his interest in the following computation.

\begin{prop} \label{Thetasq} The irreducible divisor $\Theta \subset J$ is prime exceptional, in particular, it can be contracted on some projective birational \hk model of $J$. Moreover,  $q(\Theta)=-2$. 
\end{prop}
\begin{proof}
Let $\varphi: \P_{Z'} \dashrightarrow \Theta$ be the Abel-Jacobi map as in (\ref{AJtwisted}) and let  $V \subset Z'$ be a non empty open subset such that the restriction of $\varphi$ to $r^{-1}(V)=:\P_V$ is regular. 
By restricting $V$ if necessary, we can assume that all twisted cubic parametrized by $V$ are such that $\Sigma_C$ is smooth and, in particular, that $C$ is ACM.

Recall that for any  twisted cubic $[C] \in V$, we have set $\P^1_{C}=r^{-1}(C)$. The rational curve $\varphi(\P^1_{C}) \subset J$ is smooth because it maps to $\pi(\varphi(\P^1_{C})) \subset \P^5$, which is the pencil of hyperplane sections of $X$ that contain the curve $C$. Moreover,  $(\pi \circ \varphi)(\P_V)$ intersects the dual variety $X^\vee$ in a dense open subset and, similarly, $\varphi(\P_V)$ intersects $\Theta_{X^\vee}$, the restriction of $\Theta$ to $X^\vee$,  in a dense open subset (this statement follows from \cite{Clemens} and the fact, proven there, that for a cubic threefold with one $A_1$ singularity the Abel-Jacobi mapping is birational onto its image).

We start by showing that for a general $[C] \in V$, the smooth rational curve $\varphi(\P^1_{C})$ is contained in the smooth locus of $\Theta$.
The singular locus $\Theta_{sing}$ of $\Theta$ has an irreducible component that is equal to the closure of the zero section of $J_U \to U$, while any other irreducible component of the singular locus is properly contained in $\Theta_{X^\vee}$, the restriction of $\Theta$ to $X^\vee$.
Since $V$ parametrizes ACM curves, by Remark \ref{nonACM} it follows that the intersection of $\varphi(\P_V)$ with the image of the zero section of $J \to \P^5$ is contained in $J_{X^\vee}$, the restriction of $J$ to $X^\vee \subset \P^5$.
Let $B:= \varphi^{-1}(\Theta_{sing} \cap \varphi(\P_V))$ be the locus in $\P_V$ parametrizing points mapped to the singular locus of $\Theta$ and let $W_V :=(\pi \circ \varphi)^{-1}(X^\vee \cap (\pi \circ \varphi)(\P_V) )$ be the locus in $\P_V$ of pairs $(C, Y)$ such that $Y$ is  singular. By what we have observed, it follows that $B \subseteq W_V$. 
Notice that $W_V$ is irreducible of dimension $8$, because it maps to an open subset of $X^\vee$, with fibers parametrizing equivalence classes of twisted cubics contained in a cubic threefold with one $A_1$ singularity;  these form an irreducible subset of $Z(X)$, as follows from \cite[\S 3]{Clemens}. 
We have already observed that the general point of $\Theta_{X^\vee}$ is contained in the image $ \varphi(\P_V)$. 
Thus, if $Y_0$ corresponds to a general point in $X^\vee$, there is a twisted cubic $C \subset Y_0$, with $[C] \in V$ and such that $\varphi(C,Y_0)=\phi_{Y_0}(C)$ lies in the smooth locus of $\Theta$. 
It follows that $B$ is strictly contained in $W_V$. Since $W_V$ is irreducible, $\dim B=7$. Thus $B$ does not dominate $V$ and hence the image of the open subset $\P_{V'}:=r^{-1}(V')$, where  $V':=V \setminus r(B)$ is contained in the smooth locus of $\Theta$.

For the general point in $(C, Y)  \in \P_V$, let $R:=\varphi(\P^1_{C}) \subset \Theta$ be the corresponding element of the ruling. 
By generic smoothness, the differential of $\varphi$ is of maximal rank at a general point $x \in R$, so by \cite[II. 3.4]{Kollar-rational-curves}, the vector bundle $(T_\Theta)_{|R}$ is globally generated at $x \in R$. 
It follows that $(T_\Theta)_{|R}=\oplus \mc O_R(a_i)$, with $a_i \ge 0$. 
By Lemma \ref{normalbundle} below, the restriction of the tangent bundle of $J$ to the smooth rational curve $R$ is of the form $\mc O_R^{\oplus 8} \oplus \mc O_R (2) \oplus \mc O_R(-2)$.  
Using this and the fact that $R $ is contained in the smooth locus of $\Theta$, we find that $(T_\Theta)_{|R}= \mc O_R(2) \oplus \mc O_R^{\oplus 8}$ and hence that $N_{R | \Theta}=\oplus \mc O_R^{\oplus 8}$. In particular, $\Theta \cdot R=-2$.

Consider the lattice embedding $H^2(J, \Z) \subset H^2(J,\Z)^\vee=H_2(J,\Z)$ induced by the Beauville-Bogomolov form. We claim that under this embedding, the classes of $R$ and of $\Theta$ are equal, i.e. that $R \cdot x= q(\Theta, x)$ for every $x \in H^2(J,\Z)$. 
This immediately proves the proposition, as it implies that $q(\Theta)=\Theta \cdot R=-2$.  By \cite[Cor. 3.6 (1)]{Markman-prime-exceptional} and \cite[Prop. 4.5]{Druel-exc}, the class of the ruling of a prime exceptional divisor is proportional, via a positive constant, to the class of the exceptional divisor. Thus, to prove the claim it suffices to show that $\Theta$ is prime exceptional, since the constant would have to be equal to $1$, as both $R \cdot L$ and $q(\Theta, L)$ are equal to $1$.

To prove that $\Theta$ is prime exceptional we use standard techniques on deformations of maps from rational curves to \hk manifolds, following \cite[\S 5.1]{Markman-prime-exceptional} or also \cite[\S 3]{Charles-Mongardi-Pacienza}. 
We include a proof because of  setting of Markman is different and because the proof in  \cite[\S 3]{Charles-Mongardi-Pacienza} is for projective families of \hk manifolds. 
Choose $R\subset \Theta$ a general element in the ruling and let $\Def(J)_R \subset \Def(J)$ be the smooth hypersurface in the deformation space of $J$ where the class of $R$ stays of Hodge type. 
Let $\mc {H}ilb \to \Def(J)_R$ be the component of the relative Duady space containing the point $[R]$. Since $N_{R|J}=\mc O_R(-2) \oplus \mc O_R^{\oplus 8}$, then by \cite[Thm 1]{Ran} it follows that the morphism $\rho: \mc {H}ilb \to \Def(J)_R$ is smooth at $R$ and of relative dimension $8$. 
Let $T \subset \Def(J)_R$ be a general curve containing $0$ (in particular we can assume that for very general $t \in T$, the N\'eron-Severi of the corresponding deformation $\mc J_t$ of $J$ is one dimensional and spanned by a line bundle whose class is proportional to $R_t$, the parallel transport of the class of $R$ to  $\mc J_t$) and let $\rho_T: \mc H ilb_T \to T$ be the component of the base change to $T$ of $\mc H ilb \to \Def(J)_R$ that contains $[R]$. 
Since $\rho$ is smooth at $[R]$, $\rho_T$ is dominant of relative dimension $8$. 
Up to a base change and to restricting $T$, we can assume that $\mc H ilb_T \to T$ has irreducible fibers for $t \neq 0$. 
Let $\mc J_T \to T$ be the base change of the universal family to $ T \to \Def(J)$ and let $\mc D \subset \mc J_T$ be the image of the universal family over $\mc H ilb_T$ under the evaluation map. 
Then $\mc D$ is irreducible of relative codimension $1$. Moreover, $\mc D_t$ irreducible for $t \neq 0$, and  $D:=\mc D_0$ is a union of effective uniruled divisors containing $\Theta$ as an irreducible component (with a given  multiplicity $m\ge 1$).  By the choice of $T$, for very general $t$, $\rho(\mc J_t)=1$. It follows that the class of $\mc D_t$ is proportional to the class of $R_t$ and hence that the class of $D=\mc D_0$ is  proportional to that of $R$. Moreover, the proportionality constant is positive, as both $D$ and $R$ intersect positively with a K\"ahler class. 
Hence, since $\Theta \cdot R$ is negative, so is $q(\Theta, D)$. Moreover, since the product of two distinct irreducible uniruled divisor is non negative, it follows that $q(\Theta,D) \ge m \, q(\Theta, \Theta)$.
Thus $q(\Theta, \Theta)<0$, i.e., $\Theta$ is prime exceptional. Thus, as already observed, the classes of $\Theta$ and of $R$ have to be the same and hence $q(\Theta, \Theta)=-2$.
\end{proof}

\begin{rem}
A posteriori, once we know that $\Theta$ is prime exceptional, we can use  \cite[Lem 5.1]{Markman-prime-exceptional} to show that $\mc D_0=\Theta$.
\end{rem}

\begin{lemma} \label{normalbundle}
Let $M$ be a \hk manifold of dimension $2n$ and $R \subset M$ be a smooth rational curve. Suppose $R$ is a general ruling of a uniruled divisor. Then
\[
(T_M)_{|R}=\mc O_R^{\oplus 2n-2} \oplus \mc O_R (2) \oplus \mc O_R(-2), \quad  \text{and thus} \quad N_{R|M}=\mc O_R(-2) \oplus \mc O_R^{\oplus 2n-2}
\]
\end{lemma}
\begin{proof}
Since $T_M$ is self dual, $(T_M)_{|R}=\bigoplus_i  \mc O_R (a_i) \oplus \bigoplus_i  \mc O_R (-a_i)$, for some $a_i \ge 0$.  Since $R$ is general and its deformations sweep out a divisor, by \cite[II. 3.4]{Kollar-rational-curves}, the rank of the evaluation map $\rk[H^0(R, (T_M)_{|R}) \otimes \mc O_R \to  (T_M)_{|R}]$ at a general point of $R$ is equal to $2n-1$. Hence $a_2= \dots =a_n=0$ and $a_1 \ge 2$ (cf. \cite[Prop. 4.5]{Druel-exc}). Since the normal sheaf of $R $ in $M$ is torsion free and contains the quotient $\mc O_R(a_1) \slash T_R=\mc O_R(a_1) \slash \mc O_R(2)$, it follows that $a_1=2$.
\end{proof}

Notice that the same argument as the last part of the proof of Proposition \ref{Thetasq} shows the following.

\begin{prop}
Let $M$ be a \hk manifold of dimension $2n$ and let $E \subset M$ be an irreducible uniruled divisor. Suppose that a general curve $R$ in the ruling is smooth and that $E \cdot R <0$ (e.g. if $R$ is contained in the smooth locus of $E$), then $E$ is prime exceptional and hence, under the lattice embedding $H^2(M, \Z) \subset H^2(M,\Z)^\vee=H_2(M,\Z)$ induced by the Beauville-Bogomolov form, the classes of $E$ and $R$ are proportional by a positive constant.
\end{prop}

\begin{cor} \label{J not JT}
For very general $X$, the movable cone of $J(X)$ is spanned by $L$ and $H$, where $H$ is a generator of $\Theta^\perp \subset \NS(J)$ with $q(H,L)>0$ and $q(H) >0$, i.e.
\[
\Mov (J)=\R_{\ge 0} L+\R_{\ge 0} H .
\]
In particular, there is a unique \hk model of $J$ with a Lagrangian fibration and $J$ is not birational to the twisted intermediate Jacobian fibration $J^T$.
\end{cor}
\begin{proof} We already know that one of the rays of the movable cone of $J$ is spanned by $L$.
By \cite[Thm 1.5]{Markman-prime-exceptional} the closure of the movable cone is spanned by classes that intersect non-negatively with all prime exceptional divisors. Since by Proposition \ref{Thetasq} $\Theta$ is prime exceptional, the second ray of the movable cone is determined by $\Theta^\perp$, which is spanned by a class $H$ which is big and nef on some birational \hk model of $J$. Thus,  $q(H) >0$ and $q(H,L)>0$. In particular, the movable cone is strictly contained in the positive cone implying that the only isotropic class that is movable is $L$.
\end{proof}

In terms of the other  projective \hk birational models of $J$, we can actually prove something more precise. The main result of \S \ref{section general X} describes, for general $X$, on which birational model of $J$ the proper transform of $\Theta$ can be contracted.

\subsection{Induced automorphisms}  \label{induced auto}
For \hk manifolds of K3$^{[n]}$-type, a considerable amount of literature has been devoted to the study and classification of automorphism groups. This includes studying the automorphisms induced from a K3 surface to the moduli spaces of sheaves on it. In view of Theorem \ref{hkcom}, a natural question is to study the induced action on $J$ of the automorphism group of $X$ in relation to the Lagrangian fibration structures. I thank G. Pearlstein for asking  questions that led me to the following observations.

Let $X$ be a smooth cubic fourfold and let $\tau$ be an automorphism of $X$. Then $\tau$ acts on the universal family of hyperplane sections of $X$ and thus also on the Donagi-Markman fibration $J_U \to U$ (which is identified with the relative $\Pic^0$ of the family of Fano surfaces of the hyperplane sections of $X$). By abuse of notation we denote by 
\[
\tau: J \dashrightarrow J
\]
the induced birational morphism. Notice that $\tau$ preserves $\Theta$ and $L$ so the induced action of $\tau^*$ is the identity on $U= \langle L, \Theta \rangle  \subset \NS(J)$.

\begin{prop} \label{auto}
Let $X$ be a smooth cubic fourfold and suppose that the fibers of $\pi: J \to \P^5$ are  irreducible (by  \cite{LSV} this happens for general $X$). 
\begin{enumerate}
\item Then $\Theta$ is   $\pi$-ample and so is any $B \in \NS(J)$ with $q(L,B)>0$.
\item Any  birational automorphism $\tau: J\dashrightarrow J$ which fixes $L=\pi^* \mc O(1)$ extends to a regular automorphism.
\item $L$ is nef on a unique \hk birational model of $J$. In other words, if $J' $ is a birational \hk model of $J$, with birational map $f:J' \dashrightarrow J$, and the induced map $\pi': J' \dashrightarrow J \to \P^5$ is regular, then $f$ is an isomorphism.
\end{enumerate}
\end{prop}
\begin{proof}
(1) Let $H$ be an ample line bundle on $J$ and let $J_t$ be a smooth fiber of $J \to \P^5$. Then $[H_{|J_t}]=m [\Theta_{|J_t}]$ for a positive integer $m$ so the restrictions of $H$ and $m\Theta$ are topologically equivalent for any smooth fiber. Since the fibers of $\pi$ are irreducible, it follows that  the restrictions of $H$ and $m \Theta$ to any fiber are numerically equivalent (see \cite[Lem 4.4]{Voisin-twisted}). By Nakai-Moishezon, $m \Theta$ is $\pi$-ample. Similary, if $q(B,L)>0$, then there exists positive integers $a$ and $b$ such that $aB$ and $b\Theta$ are numerically equivalent on every fiber.

(2) By assumption, $\tau^* L=L$ so $q(\tau^* \Theta, L)=q(\Theta, L)=1$. As a consequence, $\tau^* \Theta$ and $\Theta$ are topologically equivalent on the smooth fibers and hence, as above, numerically equivalent on every fiber. Thus, $\tau^* \Theta$ is $\pi$-ample. It follows that $\tau$ is a regular morphism.

(3)  Let $H'$ be any ample line bundle on $J'$ and let $L'=f^{*}L={\pi'}^*\mc O(1)$. Then $0< q(L', H')=q(L, f^*H')$, so by (1) $ f^*H'$ is ample and $f$ is an isomorphism.
\end{proof}

In addition to birational automorphisms  induced by the automorphisms of $X$, some examples of birational automorphisms which preserve $L$ are
\begin{enumerate}
\item $\iota: J \to J$ induced by the action of $(-1)$ on the smooth fibers of $J \to \P^5$.
\item $t_\alpha: J \to J$  induced by the translation of a rational section of $\alpha: \P^5 \dashrightarrow J$ (cf.\S \ref{section MW}).
\item More generally,  any birational automorphism induced by an element of the automorphism group of $J_K$, the generic fiber of $J \to \P^5$.
\end{enumerate}

\begin{rem}
As already mentioned just below Theorem \ref{LSVthm}, a necessary condition for the irreducibility of the fibers of $J \to \P^5$ is given in \cite{Brosnan}. This condition is satisfied if and only if the hyperplane sections $Y$ of $X$ satisfy $d(Y):=b_2(Y)-b_4(Y)=0$, where $b_i(Y)$ denotes the $i$-th Betti number of $Y$ and where $d(Y)$ is called the defect of $Y$. It is easy to see that if $Y$ contains a plane then $d(Y)>0$. 
\end{rem}

\section{Birational geometry of $J(X)$, for general $X$}  \label{section general X}

To describe the birational geometry of the intermediate Jacobian fibration we degenerate  the underlying cubic to the chordal cubic, following an idea already contained in \cite{KLSV}. There, it is observed that the central fiber of the corresponding family of intermediate Jacobian fibrations can be chosen to be birational to a moduli space of sheaves of OG$10$-type on a K3 surface of genus $2$. As in Section \ref{ogtentype}, by moduli space of OG$10$-type we mean a moduli space of sheaves on a K3 surface with Mukai vector $2w$, with $w^2=2$. We first refine the construction of this degeneration in order to have a central fiber that is actually \emph{isomorphic} to a certain singular moduli space of sheaves on the associated K3 surface. In this way, we can keep track of the limits of the relative theta divisor and of the line bundle inducing the Lagrangian fibration. This is done in Section \ref{deg chordal}. The results of Meachan--Zhang \cite{Meachan-Zhang}, which were recalled in Lemma \ref{MZ} imply that the central fiber of the relative theta divisor can be contracted after a Mukai flop of the zero section. For $X$ general, we then deduce the same result for $J(X)$ and, for very general $X$, we compute the nef and movable cone of $J(X)$. This is the content of Theorem \ref{NSU}. 

\begin{thm} \label{NSU} Let $X$ be a smooth cubic fourfold and let $J=J(X) \to \P^5$ be a \hk compactification of the intermediate Jacobian fibration as in \S \ref{birational section}. 

For very general $X$, 
\begin{enumerate}
\item There is a unique other \hk birational model of $J$, denoted by $N$,  which is the Mukai flop $p: J \dashrightarrow N$ of $J$ along the image of the zero section;
\item There is a divisorial contraction $h: N \to \bar N$ which contracts the proper transform of $\Theta$ onto an $8$-dimensional variety which is birational to the LLSvS $8$-fold $Z(X)$.
\end{enumerate}
In other words, we have $\Mov (J)=\langle L, H \rangle= \Nef(J) \cup p^* \Nef(N)$,  $ \Nef(J)=\langle L, H_0 \rangle$,  and  where $p^* \Nef(N)=\langle H_0 , H \rangle$, were $H_0$ is a big and nef line bundle on $J$ which contracts the zero section of $J \to \P^5$ and $H$ is as in Corollary \ref{J not JT}.

For general $X$, the relative theta divisor $\Theta$ can be contracted after the Mukai flop of the zero section of $J \to \P^5$.

\end{thm}

Before the proof of the Theorem, which will be given in the following Section  \S \ref{deg chordal}, we mention, as a consequence of the Theorem above, the relation between the intermediate Jacobian fibration and moduli spaces of objects in the Kuznetsov component of $X$.

\subsection{Comparison with moduli spaces of objects in the Kuznetsov component of $X$} \label{Kuz}

The recent  paper \cite{BLMNPS} establishes the existence and the  fundamental properties of moduli spaces of objects in the Kuznetsov component $\mc{K}u(X)$ of a smooth cubic fourfold $X$. We refer the reader to \S 29 of loc. cit for the relevant definitions and the precise statements of the results. 

Given a smooth cubic fourfold $X$, the extended Mukai lattice $\wt H(\mc K u (X), \Z)$ is a lattice, whose underlying group is the topological $K$-theory of $\mc K u (X)$ and whose Mukai pairing and  weight two Hodge structure are induced from those on $X$. The only classes in $\wt H(\mc K u (X), \Z)$  that are of type $(1,1)$ for very general $X$ are contained in a rank $2$ lattice $A_2$, which is spanned by  two classes $\lambda_1$ and $\lambda_2$, that satisfy $\lambda_1^2=\lambda_2^2=2$ and $\lambda_1 \cdot \lambda_2=-1$ (see \cite[(29.1)]{BLMNPS}). A description of a full connected component of the space of Bridgeland stability conditions on $\mc{K}u(X)$ is also produced (Thm 29.1 of loc. cit.). It is shown that, for a primitive Mukai vector with $v^2 \ge -2$ and for a $v$-generic stability condition $\sigma$ in this component, the moduli space $M_\sigma(\mc{K}u(X), v)$ of Bridgeland stable objects in  $\mc{K}u(X)$ with  Mukai vector $v$ is a non empty smooth projective \hk manifold of dimension $v^2+2$, deformation equivalent to a Hilbert scheme of points on a K3 surface; moreover, the formation of these moduli spaces works in families (see Thm 29.4 of loc. cit. for the precise statement).

For a Mukai vector of OG$10$-type in the $A_2$ lattice, i.e., of the form $v=2\lambda$ with  $\lambda^2=2$, \cite{Pertusi-et-al} shows that for a $\lambda$-generic stability condition $\sigma$ the moduli space $M_\sigma(\mc{K}u(X), v)$ is an irreducible normal projective symplectic variety of dimension $10$ admitting a symplectic resolution which is deformation equivalent to a manifold of OG$10$-type. The genericity condition here means that the polystable objects with Mukai vector $v$ are the direct sum of two stable objects with Mukai vector $\lambda$. More precisely, the singular locus of  $M_\sigma(\mc{K}u(X), v)$   is isomorphic $\Sym^2 M_\sigma(\mc{K}u(X), \lambda)$. 

Moreover, in \cite{Pertusi-et-al} it is shown that for general $X$ the twisted intermediate Jacobian fibration $J^T(X)$ is birational to $M_\sigma(\mc{K}u(X), 2\lambda)$, for $\lambda^2=2$.
For the non-twisted case we have the following corollary of Theorem \ref{NSU} that goes in the opposite direction.

\begin{cor} \label{notkuz} For very general $X$, $J(X)$ is not birational to a moduli space of the form $M_\sigma(\mc{K}u(X),v )$. \end{cor}
\begin{proof}
First of all,  by \cite[Rem 29.3]{BLMNPS}, if non empty, the dimension of a  moduli space $M_\sigma(\mc{K}u(X), v)$ is $v^2+2$. This dimension is equal to $10$ if and only if either $v$ is primitive (hence $M_\sigma(\mc{K}u(X), v)$ is of K3$^{[5]}$-type and thus cannot be birational to $J(X)$) or else $v=2\lambda$ with $\lambda^2=2$.
By the results of \cite{Pertusi-et-al} cited in the remark above, for $v=2\lambda$ with $\lambda \in A_2$ and $\lambda^2=2$ and $\sigma$ a $\lambda$-generic stability condition, the singular locus of $M_\sigma(\mc{K}u(X), v)$ is isomorphic the second symmetric product of a  \hk manifold of K3$^{[2]}$--type. By dimension reasons, the symplectic resolution $\wt M_\sigma(\mc{K}u(X),v ) \to M_\sigma(\mc{K}u(X), v)$ is not a small contraction. Suppose by contradiction that $J(X)$ is birational to $\wt M_\sigma(\mc{K}u(X), v)$. Then by Theorem \ref{NSU},  the symplectic resolution has to coincide with $N \to \overline N$ and $M_\sigma(\mc{K}u(X), v) \cong \bar N$. This implies that the singular locus of $M_\sigma(\mc{K}u(X), v)$ has to be birational to the Lehn-Lehn-Sorger-van Straten $8$-fold $Z(X)$, which gives a contradiction. Indeed, $Z(X)$ cannot be birational to $\Sym^2 M_\sigma(\mc{K}u(X), \lambda)$ since, by Proposition \ref{GLR}, this would imply that the latter has a symplectic resolution. This, however, is not true because $\Sym^2 M_\sigma(\mc{K}u(X), \lambda)$ is a $\Q$-factorial sympectic variety with singular locus of codimension strictly greater than $2$ and hence does not admit a symplectic resolution (since it does not admit a semi-small resolution).
\end{proof}

\begin{rem} We expect the more general statement to hold: for very general $X$, $J(X)$ is not birational to a Bridgeland moduli space of objects on a $2$-CY category that is deformation equivalent to the derived category of a K3 surface. We present a rough sketch of the argument. Assume there is a family of Bridgeland stability conditions on the family of derived categories realizing the deformation. Then as in  \cite[21.24]{BLMNPS}  a relative moduli space exists as an algebraic space; by a generalization of a theorem of Mukai \cite[Thm 1.4]{Perry} the stable locus of each fiber is smooth and has a holomorphic symplectic form; the singular locus parametrizing strictly semi-stable objects of codimension $\ge 2$.  One then expects such moduli spaces to be normal and irreducible. As in the proof of the projectivity in \cite[Thm 29.4]{BLMNPS} then shows that these moduli spaces are projective. Finally, a similar argument to the one above shows that the contraction $N \to \bar N$ cannot be the symplectic resolution of one of these moduli spaces.
\end{rem}


In the next subsection we construct the degeneration of the intermediate Jacobian fibration that will allow us to prove Theorem \ref{NSU}. The proof of the Theorem will be given at the end of the Section.

\subsection{Degeneration to the Chordal Cubic} \label{deg chordal}

The secant variety to the Veronese embedding of $\P^2$ in $\P^5$ is a cubic hypersurface isomorphic to $\Sym^2 \P^2$, called the chordal cubic. Such a singular cubic fourfold is unique up to the action of the projective linear group.
Given a one parameter family of cubic fourfolds degenerating to the chordal cubic, it was proved  in \cite{KLSV}  that, up to a base change, one can fill the corresponding degeneration of intermediate Jacobian fibrations with a smooth central fiber that is birational to  $\wt M_{2v_0}=\wt M_{2v_0}(S)$, where $(S,C)$ is the degree $2$ $K3$ surface associated to the degeneration of cubic fourfolds  as in \cite{Collino-fundamental,Hassett-special, Laza} and where $v_0=(0,C,-2)$ is as in  (\ref{vkappa}). We will use this degeneration to study the birational properties of the intermediate Jacobian fibration, at least for general $X$. For this purpose, we need to control what happens to the line bundles $L$ and $\Theta$ under the corresponding degeneration of intermediate Jacobian fibrations. We achieve this by constructing a particular degeneration whose central fiber is precisely the singular moduli space $M_{2v_0}$ and  is such that the Lagrangian fibrations of the members of this degeneration fit in a relative Lagrangian fibration. This is done in Proposition \ref{3fams}. With this degeneration, we are not only able to identify precisely the limits of $L$ and $\Theta$ (see Lemma \ref{Udeforms}), but we are also able to deform the results about the birational geometry of $M_{2v_0}$  away from the central fiber (see Proposition \ref{propmukflop}), eventually proving Theorem \ref{NSU}.

Let $\mc X \to \Delta$ be a one parameter family of cubic fourfolds degenerating to the chordal cubic. By this we will mean $\Delta$ is a small disk or an open affine subset in the base of a pencil of cubic fourfolds with the property that the general fiber is smooth and the central fiber is isomorphic to the chordal cubic.
Recall the following facts (proved in \cite{Hassett-special}, cf. also \cite{Laza} and \cite{KLSV}): (a) the monodromy of this family  has order $2$; (b) to such a degeneration one can associate a degree $2$ polarized K3 surface $(S,C)$; (c) for a general pencil, the polarized K3 surface $(S,C)$ is general in moduli.

Suppose that for $t \neq 0$ the cubic fourfold $\mc X_t$ is general in the sense of LSV (i.e. in the sense that the construction of the \hk compactification of \cite{LSV} works for $J_U(\mc X_t)$) and let $\mc J^* \to \Delta^*$ be the family of intermediate Jacobians associated to the smooth locus  $\mc X^* \to \Delta^*$ of the pencil, with corresponding family of Lagrangian fibrations $\pi_{\Delta^*}:\mc J^* \to \P^5 _ {\Delta^*}$. 

\begin{lemma}[\hspace{1sp}\cite{Collino-fundamental, KLSV}] \label{collino} 

Up to a degree $2$ base change, we can extend $\pi_{\Delta^*}:\mc J^* \to \P^5 _ {\Delta^*}$ to a projective morphism $\pi_{\mc V}: \mc J_{\mc V} \to \mc V$, where $\mc V \subset \P^5 \times \Delta$ is an open subset such that $\mc V_t=\P^5$ for $t \neq 0$ and $ \mc V_0 \subset \P^5$ is non empty for $t=0$, and where $\mc { J}_0 \to \mc V_0 \subset \P^5$ is identified with the restriction of $M_{2v_0}(S) \to |2C|=\P^5$ (cf. \ref{lagrfibr})  to an open subset $V \subset |2C|$. Moreover, $\mc J_{\mc V} \to \mc V$ has a zero section and is polarized by a relative principal polarization.
\end{lemma}
\begin{proof} Let $H \subset \P^5$ be a general hyperplane. For the degeneration $\mc Y:= \mc X \cap (H \times \P^5)$ of a single  smooth cubic $3$-fold the statement is due to Collino \cite{Collino-fundamental}. In Prop. 1.16 of loc. cit, it is also shown that the class of the limit polarization is the theta divisor of the Jacobian of the genus $5$ hyperelliptic curve, which is the limiting abelian variety. For the statement about the limit of the intermediate Jacobian fibration, this is \cite[\S 6.3 ]{KLSV}.
\end{proof}

We now compactify the projective family $\mc J_{\mc V}$ of the Lemma above to construct a family of Lagrangian fibered holomorphic symplectic varieties in such a way that the central fiber is exactly $M_{2v_0}=M_{2v_0}(S)$ (or $\wt M_{2v_0}=\wt M_{2v_0}(S)$) (cf. (\ref{lagrfibr})).

\begin{prop} \label{3fams} Let $\mc X \to \Delta$ be as above a general family of smooth cubic fourfolds degenerating to the chordal cubic. Suppose that for very general $t \in \Delta$, $\mc X_t$ is very general. Let $(S,C)$ be the corresponding K3 surface of degree $2$ as above. Then, possibly up to a base change, there are two degenerations of the corresponding intermediate Jacobian fibration, fitting in the following commutative diagram
\be \label{diagr3fam}
\xymatrix{
\wt{\mc{M}}  \ar[r]^m \ar[dr]_{\wt f} &  \mc{M} \ar[d]^{f}  \\
& \Delta 
}
\ee
where:

\begin{enumerate}
\item $\wt f: \wt{\mc{M}} \to \Delta$ is a family of smooth \hk manifolds, with $\wt{\mc{M}}_t=J(\mc X_t)$ for $t \neq 0$ and $\wt{\mc{M}}_0=\wt M_{v_0}(S)$. The family is equipped with a relative Lagrangian fibration $ \wt{\mc{M}} \to \P^5_\Delta$, where for each $t$ the corresponding Lagrangian fibration is the obvious one.
\item $f: \mc M \to \Delta$ is a degeneration of \hk manifolds, with ${\mc{M}}_t=J(\mc X_t)$ for $t \neq 0$ and ${\mc{M}}_0= M_{v_0}(S)$. The morphism $m: \wt{\mc{M}} \to \mc M$ is proper, birational, for $t \neq 0$ it is an isomorphism and  for $t=0$ it is the natural symplectic resolution $m_0: \wt M_{2v_0}(S) \to M_{2v_0}(S)$ of Theorem \ref{OG10}. Moreover,  there is a relative Lagrangian fibration ${\mc{M}} \to \P^5_\Delta$ where for each $t$ the corresponding Lagrangian fibration is the obvious one.
\end{enumerate}
\end{prop}

\begin{proof} Start from the projective morphism $\pi_{\mc V}:\mc J_{\mc V} \to \mc V$ of Lemma \ref{collino}. There is an isomorphism $\mc J_{\mc V_0}\cong  (\wt M_{2v_0})_V$ , where $ (\wt M_{2v_0})_V$ is the restriction of the Lagrangian fibration $\wt \pi:  \wt M_{2v_0} \to \P^5$ (cf. (\ref{lagrfibr})) to an open subset $V \subset \P^5$. 
Let $\overline{\mc J_{\mc V}} \to \P^5_\Delta$ be any projective morphism extending $\pi_{\mc V}$. Applying Theorem \ref{deg lagr 1} (2) to $\overline{\mc J_{\mc V}} \to \P^5_\Delta$ yields, possibly up to the base change, a family $\wt g: \wt{\mc J} \to \Delta$ of smooth \hk manifolds (projective over $(\Delta)^*$), with a relative Lagrangian fibration $\wt{\mc J} \to \P^5_{\Delta}$. Let  $\mc L$ be the line bundle on $\wt{\mc J}$ inducing it on every fiber. Let
\be
\phi_0:  \wt{\mc J}_{0} \dashrightarrow \wt M_{2v_0}
\ee
be the birational morphism induced by the isomorphism of open subsets $\mc J_{\mc V_0} \cong (\wt M_{2v_0})_V$. Then $(\phi_0)_* \mc L_0=\wt \ell:=\wt \pi^* \mc O_{\P^5}(1)$.

We now use an argument very similar to that in the proof of  \cite[Thm 1.3 and 1.7]{KLSV}, to construct a family which is isomorphic to $\mc J$ over $\Delta$ and whose central fiber is actually \emph{isomorphic} to $ \wt M_{2v_0}(S)$.
Let $\Lambda$ be the OG$10$ lattice. Fixing a marking of the central fiber and trivializing the local system $R^2 \wt g_* \Z$ induces a marking $\eta_t: H^2(\wt{\mc J}_t, \Z) \to \Lambda$ of every fiber. Let  $\mc D \subset \P (\Lambda \otimes_\Z \C)$  be the period domain and let $\mc P: \Delta \to \mc D$ be the period mapping induced by these markings.
Let $\rho_0= \eta_0 (\phi_0)^*:  H^2( \wt M_{2v_0}, \Z) \to \Lambda$ be the induced marking on $\wt M_{2v_0}$.  
Let $\rho_t: H^2(\wt{\mc M}_t, \Z) \to \Lambda$ be markings induced by $\rho_0= \eta_0 (\phi_0)^*$ on fibers of the universal family over  $\Def(\wt M)$ and let $\mc P_{\wt M}: \Def(\wt M) \to \mc D$ be the induced period mapping. Since $\mc P_{\wt M}$ is a local isomorphism, we can lift $\mc P$ to a map $ \xi: \Delta \to \Def(\wt M_{2v_0})$. Pulling back the universal family gives a family $\wt f: \wt{\mc M} \to \Delta$ with central fiber $\wt{\mc M}_0=\wt M_{2v_0}$. As in \cite{KLSV} the two families $\wt g: \wt{\mc J} \to {\Delta}$ and $\wt f: \wt{\mc M} \to \Delta$ are relatively birational over $\Delta$, since for every $t \in \Delta$, the marked pairs $(\wt{\mc J}_{t}, \eta_t)$ and $(\wt{\mc M}_t, \rho_t)$ are non separated points.
To show that the two families $\wt{\mc J}$ and $\wt{\mc M}$ are isomorphic away from the central fiber, first recall that by \cite[Thm 4.3]{Huybrechts-compact-hk} (cf. also \cite[Thm 3.2]{Markman-Survey}), for every $t$ there exists an effective cycle
\[
\Gamma_t = Z_t+ \sum W_{i,t}
\]
of pure dimension $10$ in $\wt{\mc M}_t \times \wt{\mc J}_t$ such that:  $Z_t$ is the graph of a birational map;  the codimension of the images of the $W_{i,t}$ in  $\wt{\mc M}_t $ and in $ \wt{\mc J}_t$ are equal and positive;  $[\Gamma_t]_*$ is a Hodge isometry and is equal to $\rho_t^{-1} \circ \eta_t: H^2(\wt{\mc J}_t, \Z) \to H^2(\wt{\mc M}_t, \Z)$.
Let $\wt{\mc L}$ be the line bundle on $\mc M$ such that $\wt{\mc L}_t=\rho_t^{-1} \eta_t ({\mc L})=[\Gamma_t]_*(\mc L_t)$. Since $\wt{\mc L}_0=\wt \pi^* \mc O_{\P^5}(1)$ induces a Lagrangian fibration on $\wt{\mc M}_0=\wt M_{2v_0}$, by \cite{Matsushita-Def} $\wt{\mc L}$ induces a Lagrangian fibration on $\mc M_t$ for every $t$ (maybe up to restricting $\Delta$).
For very general $t$, $\wt{\mc L}_t=[Z_t]_*(\mc L_t)$,  since the isotropic class $[Z_t]_*(\mc L_t)$ lies in the movable cone of $\wt{\mc M}_t$ and hence by Corollary \ref{J not JT} it has to be equal to $\wt{\mc L}_t$. 
Corollary \ref{onlyJ} implies that for very general $t$, $Z_t$ is the graph of an \emph{isomorphism} between $\wt{\mc J}_t$  and $\wt{\mc M}_t$. 
The same countability argument as in the proof of \cite[Thms 1.3 and 1.7]{KLSV} shows that there exists a component of the Hilbert scheme parametrizing  graphs  of such cycles $Z_t \subset \mc J_t \times \mc M_t$ that dominates $\Delta$. It follows that there is a cycle $\mc Z$ in the fiber product $\wt{\mc J} \times_\Delta \wt{\mc M}$ which, maybe up to restricting $\Delta$, induces an isomorphism for $t \neq 0 $.
The conclusion is that the family  $\wt{\mc M} \to \Delta$ is such that central fiber is $\wt{\mc M} _0 \cong \wt M_{2v_0}$ while for $t \neq 0$, we have  $\wt{\mc M}_t \cong J(\mc X_t)$.

Now we construct the second family.
By \cite[Thm 2.2]{Namikawa-def} there is a finite morphism
\[
\Xi:  \Def(\wt M_{2v_0}) \to  \Def( M_{2v_0})
\]
 induced by the symplectic resolution $m_0: \wt M_{2v_0} \to M_{2v_0} $  and compatible with the  universal families on the two deformation spaces (for more details see \S 2 of loc. cit.).  Set $\nu=\Xi \circ \xi: \Delta \to \Def( M_{2v_0})$ and let
\[
\mc M \to \Delta
\]
be the pullback via $\nu$ of the universal family via $\nu$ on $\Def( M_{2v_0})$. Then the birational map $m: \wt{\mc M} \to \mc M$ over $\Delta$ induced by \cite[Thm 2.2]{Namikawa-def} has the desired properties.

Finally, the statement about the Lagrangian fibrations follows from the fact that, since the Lagrangian fibration $\wt M_{2v_0} \to \P^5$ in the central fiber factors via $\wt M_{2v_0} \to M_{2v_0}$, the morphism $\wt {\mc M} \to \P^5_{\Delta}$ factors via a morphism $\mc M \to  \P^5_{\Delta}$.
\end{proof}

As a consequence of the last part of the proof, notice that there is a line bundle $\mc L_{\mc M}$ on ${\mc M}$ with
\[
m^* \mc L_{\mc M}=\wt{\mc L}
\]
and such that its restriction to the central fiber satisified $\mc L_{{\mc M}_0}=\ell$, where $\ell$ is as in (\ref{ell}). 

For any $t \neq 0$, let $ \Theta_t$  be the relative Theta divisor in $ \mc M_t=J(\mc X_t)$.

\begin{lemma} \label{lemmacomp}
For $\star=\wt{\mc M} $ or $\mc M$, let  $\Theta_\star$ be the divisor defined as the closure  of  $ \cup_{t \neq 0} \Theta_t$ in $\star$. 
Then,   $\Theta_{{\mc M}}$ is a Cartier divisor and hence the following compatibility conditions hold (notation as in diagram (\ref{diagr3fam})):
\be \label{compatibility}
\Theta_{\wt{\mc M}_0} \cong (m^*\Theta_{\mc M})_{|\wt{\mc M}_0} =m_0^*\Theta_{\mc M_0}
\ee
where $\Theta_{\star_0}:=(\Theta_\star)_{|0}$ is the fiber of $\Theta_\star$ over $t=0$. 
\end{lemma}

\begin{proof} Let $\mc I_{\Theta_{\mc M}} \subset \mc O_{{\mc M}}$ be the ideal sheaf of $\Theta_{\mc M}$ in ${\mc M}$. Since the morphism $\Theta_{\mc M} \to \Delta$ is flat, it follows that the restriction $(\mc I_{\Theta_{\mc M}} )_{|\mc M_0}$ is the ideal sheaf of $\Theta_{\mc M_0}$ in ${\mc M_0}$. By \cite{Perego-Rapagnetta} (cf. \S \ref{movog10}), $\mc M_0=M_{2v_0}$ is factorial so $(\mc I_{\Theta_{\mc M}} )_{|\mc M_0}$ is locally free. Hence, so is $\mc I_{\Theta_{\mc M}} $. It follows that the divisors $\Theta_{\wt{\mc M}}$ and $m^*\Theta_{\mc M}$ agree and so do their central fibers. 
\end{proof}

The next Lemma identifies the limit of $\Theta_t$ in $M_{2v_0}=\mc M_0$ and shows that all line bundles on $M_{2v_0}$ deform over $\mc M \to \Delta$. Recall first that by (\ref{NSsing}), $\NS(M_{2v_0})=U=\langle \ell, \theta \rangle$ and that for every $t$
\[
\NS(\mc M_t) \supseteq U_t=\langle \mc L_t, \Theta_t \rangle
\]
equality holding for very general $t$. Here $\Theta_0=\Theta_{\wt{ \mc M_0}}$. In particular, inside $\NS(\wt M_{2v_0})$ we have the following rank $2$ sublattices both of which are isomorphic to the hyeperbolic lattice $U$: The limit lattice $U_0$ spanned by the limits $\mc L_0=\wt \ell$ and $\Theta_0$ and the pullback lattice $m_0^*\NS(M_{2v_0})=\langle m_0^* \ell, m_0^*\theta \rangle$.

\begin{lemma} \label{Udeforms} Let the notation be as above. Then
\begin{enumerate}
\item The two  sublattices $U_0=\langle \wt \ell,  \Theta_{\wt{ \mc M_0}}\rangle$ and $\langle m_0^* \ell, m_0^*\theta \rangle$ of $\NS(\wt M_{2v_0})$ are the same.
\item The limit of the relative Theta divisor in $\mc M_0$ is precisely  $\theta$, the relative theta divisor on $M_{2v_0}(S)$ of (\ref{deftheta}).
\end{enumerate}
\end{lemma}
\begin{proof}
By Lemma \ref{collino} and Lemma \ref{lemmacomp}, the limit theta divisor $(\Theta_{\mc M})_0$ is an effective line bundle  on $\mc M_0=M_{2v_0}$ which restricts to a theta divisor on the smooth fibers of  $M_{2v_0} \to \P^5$. Thus $\Theta_{\mc M_0}$ is linearly equivalent to an effective line bundle of the form $\theta+a \ell$ for some integer $a$.  We  show that $a=0$. By (\ref{compatibility}), $ \Theta_{\wt{\mc M}_0}  =m_0^*\Theta_{\mc M_0}=m_0^*(\theta+a \ell)$ and $\wt{\mc L}_0=m_0^*\ell $. This is enough to conclude that the two sublattices
\[
U=\langle \Theta_{\wt{\mc M}_0}, \wt{\mc L}_0  \rangle \quad  \text{ and } U=m_0^* \langle  \theta, \ell \rangle=m_0^* \NS(M_{2v_0})
\]
 of $\NS (\wt M_{2v_0})$ are the same. 
 This proves the first part of the Lemma. 
By Remark \ref{Alexeev} above, $\theta$ does not contain the singular locus of $M$, thus $m_0^* \theta$ coincides with its proper transform and is  irreducible. 
Since it has negative Beauville-Bogomolov square (cf. (\ref{thetasq})), it is a prime exceptional divisor.  
By \cite[\S 5.1]{Markman-prime-exceptional}, a prime exceptional divisor deforms where its first Chern class remains algebraic. 
Thus  $m_0^* \theta$ deforms to relative effective prime exceptional divisor $\theta_{\wt{\mc M} }$ on $\wt{\mc M}$. 
 By Corollary \ref{onlyJ} and Proposition \ref{Thetasq}, for very general $t \neq 0$, the fiber over $t$ of the two irreducible effective divisors $\theta_{\wt{\mc M} }$ and $\Theta_{\wt{\mc M}}$ have to agree since there is only one prime exceptional divisors on $\mc M_t$. Thus $\theta_{\wt{\mc M} }$ and $\Theta_{\wt{\mc M}}$ have to be equal for every $t$. In particular, so are their restrictions to the central fiber.
 \end{proof}

\begin{cor} \label{zerosectgenX}
Let $X$ be a general cubic fourfold and let $\pi: J(X) \to \P^5$ be the intermediate Jacobian fibration of \cite{LSV}. The natural rational zero section of $\pi$ is regular.
\end{cor}
\begin{proof} Consider a degeneration of cubic fourfolds to the chordal cubic as in Proposition \ref{3fams} and let $\mc M \to \Delta$ be the corresponding family. By Lemma \ref{lemmacomp} the divisor $\Theta_{\mc M}$ is Cartier and by Remark \ref{thetaamp}  it is relatively ample (up to restricting $\Delta$). Since $\Theta_{\mc M}$ is $-1$-invariant, it follows that the birational involution $-1$ is biregular. One component of the fixed locus of this involution has the property that its restriction to every fiber is precisely the closure of the corresponding rational zero section. Since by Remark \ref{zerosect} in the central fiber the section is regular, it follows that for general $t \in \Delta$ the corresponding rational section is also regular.
\end{proof}

Consider the family  ${\mc M} \to \Delta$ of Proposition \ref{3fams}, with its relative theta divisor $\Theta_{{\mc M}}$. By Druel \cite{Druel} we know that for every $t$, the prime exceptional divisor $\Theta_t$ can be contracted on a \hk projective birational model of ${\mc M} _t$. In the central fiber ${\mc M} _t=M_{2v_0}$ we have, by Lemma \ref{Udeforms}, that $\Theta_{{\mc M_0}}=\theta$. By Lemma \ref{MZ} this divisor can be contracted after a Mukai flop. We now show that the same is true for any $t \neq 0$, namely, that after a Mukai flop the relative theta divisor can be contracted, possibly up to restricting $\Delta$.

\begin{prop} \label{propmukflop} For general $X$, the relative theta divisor $\Theta$ on $J=J(X)$ can be contracted after the Mukai flop of the zero section.
\end{prop}

\begin{proof}
Let ${\mc M} \to \P^5_{\Delta}$ be as in Proposition \ref{3fams}. By Corollary \ref{zerosectgenX}, there is a relative zero section $s: \P^5_\Delta \to \mc M$. Let $T$ be its image. Then $T$ is contained in the smooth locus of  the fibers of $\wt{\mc M} \to {\Delta}$. Let 
\[
P: \mc M \dashrightarrow \mc N
\]
be the relative Mukai flop of $T$  in ${\mc M}$. By Proposition \ref{MZ},  the Mukai flop of the zero section in the central fiber  $M_{2v_0}$ can be preformed in the projective category. Thus, the central fiber of $\mc N$ is projective and so are all the fibers of $g: \mc N \to \Delta$ (since by Lemma \ref{Udeforms} there is an ample class on the central fiber that deforms over $\Delta$). For $t \neq 0$, $\mc N_t$ is smooth while the central fiber $\mc N_0 $ has the same singularities as $\mc M_0=M_{2v_0}$, since they are isomorphic away from the flopped locus which does not meet the singular locus. 
 Via the birational morphism $P$, which is a relative isomorphism in codimension $1$, we can identify the second integral cohomology group of the fibers of the two families. 
 In particular, for every $t \in \Delta$ we have $P_*U_t \subset \NS(\mc N_t)$ with equality holding for very general $t$ and for $t=0$.  In what follows we freely restrict $\Delta$, if necessary, without any mention.   
 As in Propostion \ref{MZ}, let $H$ be the big and nef line bundle on $\mc N_0$ that contracts $\theta$ (i.e. $H$ is a generator of the ray $\theta^\perp$). Since $H \in P_*U_0$, by Lemma \ref{Udeforms}, $H$ deforms to a line bundle $\mc H$ on $\mc N$. For very general $t$, its restriction $\mc H_t$ is a generator of the one dimensional space $(P_t)_*\Theta_t^ \perp \subset \NS(\mc N_t)$. 
By \cite{Druel-exc}, $(P_t)_*\Theta_t$ can be contracted on a birational model of $\mc N_t$. We now show that it can be contracted on $\mc N_t$ itself. For very general $t$, the line bundle inducing the divisorial contraction has to be $\mc H_t$, or rather its proper transform on an appropriate birational model of $\mc N_t$. It follows that for very general $t$ (and thus for all $t$) $\mc H_t$ is big. Moreover, since $\mc H_0$ is big and nef and $\mc N_0$ has rational singularities, $H^i(\mc N_0, \mc H_0^k)=0$ for $i>0$ and any $k \ge 0$. It follows that  the locally free sheaf $g_* \mc H^k$ satisfies base change. Since $\mc H_0$ is semi-ample,  so is $\mc H_t$ for all $t$ in $\Delta$. For $k \gg 0$, the regular morphism $\Psi: \mc N \to \P (g^*g_* \mc H^k)$, relative over $\Delta$,  is birational onto its image and  contracts $\Theta_t$ for very general $t$ and for $t=0$. Up to further restricting $t$, we can assume that the locus contracted on $\mc N_t$ is irreducible, and hence that $\Psi_t$ contracts precisely $(P_t)_*\Theta_t$ for every $t$.

\end{proof}

The proof of Theorem \ref{NSU} is now complete:

\begin{proof}[Proof of Theorem \ref{NSU}]  Let $X$ be  general. By Proposition \ref{propmukflop},   the Mukai flop $p: J \dashrightarrow N$ of $J$ along the zero section is projective and on $J$ there exists a  big and nef line bundle  that contracts the zero section.  For very general $X$, $H_0$ is unique, up to a positive rational multiple and $\Nef(J)=\langle L, H_0 \rangle$. Moreover, we have showed that for general $X$ there is a divisorial contraction $N \to \bar N$, contracting $p_*\Theta$. Since the divisorial contraction $N \to \ov N$ contracts the ruling of $\Theta$ (cf. Prop. \ref{Thetasq}), by (\ref{AJtwisted}) it follows that the image of $\Theta$ in $\bar N$ is birational to the LLSvS $8$-fold $Z(X)$.
 For very general $X$, $\Nef(N)=\langle p_*H_0, p_*H \rangle$, where $p^*H$ is the unique (up to a positive multiple) big and nef line bundle inducing the contraction.
By \cite[Prop. 4.2]{Huybrechts-kahler-cone}, $H$ is the second ray of the movable cone of $J$, i.e. $\Mov(J)=\langle L, H \rangle$. 
\end{proof}


\section{The Mordell-Weil group of $J(X)$} \label{section MW}

Let  $a: \mc A \to B$ be a projective family of abelian varieties over an irreducible basis $B$ and suppose that $a$ admits a zero section.  The Mordell-Weil group $MW(a)$ of $a: \mc A \to B$ is  the group of rational sections of $a: \mc A \to B$. Equivalently, if $K$ denotes the function field of $B$, $MW(a)$ is the group of $K$-rational points of the generic fiber $\mc A_K$. For Lagrangian \hk manifolds, the study of the Mordell-Weil group of abelian fibered \hk manifolds was started by Oguiso in \cite{Oguiso-shioda,Oguiso-MW}. 
The aim of this section is to prove the following theorem

\begin{thm} \label{MW}
Let $X$ be a smooth cubic fourfold and  let $\pi: J=J(X) \to \P^5$  be as in Theorem \ref{hkcom}, a smooth projective \hk compactification of $J_U$. Let $MW(\pi)$ be the Mordell-Weil group of $\pi$, i.e., the group of rational sections of $\pi$ and let $H^{2,2}(X, \mathbb Z)_0$ be the primitive degree $4$ integral cohomology of $X$.
 The natural group homomorphism 
\[
\phi_X: H^{2,2}(X, \mathbb Z)_0 \to MW(\pi)
\]
induced by the Abel-Jacobi map (cf. \ref{AbelJacobi}) is an isomorphism.
\end{thm}

\begin{cor} The group $MW(\pi)$ is torsion free.
\end{cor}

\begin{rem} \label{rem og} In \cite{Oguiso-MW} Oguiso proved the existence of Lagrangian fibered \hk manifolds whose Mordell-Weil group has rank $20$. This is the maximal possible rank among all the known examples of \hk manifolds, as follows from the Shioda-Tate formula of \cite{Oguiso-shioda} (see also Proposition \ref{shioda-tate} below). Oguiso considers deformations of the abelian fibration $\wt M_{2v_0} \to \P^5$ (cf. \ref{lagrfibr}) preserving both the Lagrangian fibration structure and the zero section; among these deformations, Oguiso shows the existence of Lagrangian fibration with rank $20$ Mordell-Weil group  \cite[Thm 1.4 (2)]{Oguiso-shioda}.
The general deformation of $\wt M_{2v_0} \to \P^5$ for which both the Lagrangian fibration structure and the zero section are preserved (this is a codimension two condition) is, up to birational isomorphism,  $J(X)$ (see Remark \ref{section deform}). By the Theorem above Lagrangian fibrations of the form $J(X)$, for $X$ with $\rk H^{2,2}(X, \mathbb Z)=22$, satisfy $\rk MW(\pi)=20$. Thus, they provide an explicit description of Oguiso's examples.
\end{rem}

The following Proposition is essentially a reformulation of  results from \cite{Oguiso-shioda,Oguiso-MW}.

\begin{prop} \label{shioda-tate}
Let $\pi: M \to \P^n$ be a projective \hk manifold with a fixed (rational) section. Let $K=\C(\P^n)$ be the function field of the base and let $M_K$ the base change of $M$ to the generic point of $\P^n$.
There is a commutative diagram,
\[
\xymatrix{
 &  & 0 \ar[d]  & 0 \ar[d] & \\
0 \ar[r]  & \Z L  \oplus \bigoplus_i \Z D_{i} \ar[r] \ar@{=}[d] & L^\perp  \ar[d] \ar[r]& \Pic^0(M_K)  \ar[r] \ar[d]  & 0 \\
 0 \ar[r]  &\Z L \oplus \bigoplus_i \Z D_{i} \ar[r]  & \NS(M)  \ar@{->>}[d]^{r_b} \ar[r]^{r_K} & \Pic(M_K) \ar[r] \ar@{->>}[d] & 0 \\
 &  & \Z  \ar@{=}[r]  & NS(M_K) &
}
\]
where $L=\pi^* \mc O_{\P^n}(1)$ and where the $D_{1}, \dots ,D_k$ are the irreducible components of the complement of the regular locus of $\pi$ that do not meet the section. In particular, $\rk (MW(\pi))=\rk (\NS(M))-\rk  \Z L \oplus \Z D_{i}-1=\rk (\NS(M))-k$.
\end{prop}
\begin{proof}
The column on the left is exact by definition. By \cite{Voisin-lagr}, for $b$ in the  locus $U \subset \P^n$ parametrizing smooth fibers of $\pi$, $\im[r_b: \NS(M) \to H^2(M_b)]=\Z$. The same argument as in Lemma \ref{LTheta} shows that a line bundle $D$ on $M$ lies in $L^\perp$ if and only if $D^n \cdot L^n=(D_{|M_b})^n=0$. Since $\rk r_b=1$, this holds if and only if $D \cdot L^n=D_{|M_b}=0$, which is equivalent to $D \in \ker r_b$. This shows that the central column is exact. The same argument of \cite[Thm 1.1]{Oguiso-MW}, which was used to show that $\rk \NS(M_K)=1$, shows that any element in $\ker(r_b)=L^\perp$ goes to zero in $\NS(M_K)$. Thus there are induced horizontal morphisms $L^\perp \to  \Pic^0(M_K) $ and $\Z \to \NS(M_K) $. Since $\NS(M) \to \Pic(M_K) $ is surjective, the bottom horizontal morphism is an isomorphism. The natural morphism  $\Z L \oplus \Z D_{i} \to \NS(M)$ is injective, since by \cite[Lem 2.4]{Oguiso-shioda} it has maximal rank over $\Q$ and $\NS(M)$ is torsion free. Clearly, $ \Z L \oplus_i \Z D_{i}  \subset \ker(r_K)$. To show the reverse inclusion, let $D$ be any line bundle on $M$ that goes to zero in $\Pic(M_K)$. Then, by what we have already proved, for any smooth fiber we have $r_b(D)=[D_{|M_b}]=0$. It follows that $D$ is a linear combination of $L=\pi^* \mc O_{\P^n}(1)$ and boundary divisors, i.e., $D \in \Z L \oplus_i \Z D_{i}$. Since $\rk(MW(\pi))=\rk  \Pic^0(M_K)$, the last statement also follows.

\end{proof}

\begin{rem}
The study of the Mordell-Weil group for the Beauville-Mukai system is being carried out in a joint work in progress with Chiara Camere. 
\end{rem}

\begin{cor} \label{ranks}
Let $J=J(X) \to \P^5$ be a \hk compactification of the intermediate Jacobian fibration. Then
\[
\rk MW(\pi)=\rk \NS(J)-2=\rk H^{2,2}(X,\Z)_0.
\]
\end{cor}
\begin{proof}
The discriminant locus of $\pi$ is irreducible and the fibers of $\pi$ over the general point of the discriminant are also irreducible (cf. Lemma \ref{lem Juno}). Thus, in the notation of Proposition above, $\ker r_K=\Z L$ and the equality $\rk MW(\pi)=\rk \NS(J)-2$ follows. The remaining equality follows from Lemma \ref{tr lattice}.

\end{proof}

\begin{rem} The corollary just proven, which relies on Oguiso's Shioda-Tate formula above, is the only part of this section where we use that $J_U$ admits a \hk compactification with a regular Lagrangian fibration extending $J_U \to U$. Indeed, to define the Abel-Jacobi map $\phi_X$ and to prove that it is injective (\S \ref{sec injectivity}), we don't need to assume the existence of a \hk compactification. However, we will use this Corollary in the proof of the surjectivity  (\S \ref{sec surjectivity}).
\end{rem}

\begin{rem} An interesting problem is to study the action on $J$ of the birational automorphisms induced by translation by a non-trivial element of $MW(\pi)$ as well as to study the automorphism group of the generic fiber $J_K$.
Notice that, as a consequence of the observations of \S \ref{induced auto}, if $J \to \P^5$ has irreducible fibers then the birational automorphisms induced by translation are regular morphisms.
\end{rem}

\subsection{The Abel-Jacobi mapping} \label{subsection AJ}
This sections uses some ingredients from the theory of normal functions (certain holomorphic sections of intermediate Jacobian fibrations), as developed and used by Griffiths \cite{Griffiths-periods,Griffiths-periods-III}, Zucker \cite{Zucker-inventiones,Zucker-HC}, and Voisin \cite{Voisin-Some-Aspects}. We refer to these papers, as well as to  \cite[\S 7.2.1, 8.2.2]{Voisin2} for the relevant theory.

The first task is to define the morphism $\phi_X: H^{2,2}(X,\Z)_0 \to MW(\pi)$. One way to do this is to use relative Deligne cohomology, which allows to define an algebraic section of the fibration $J_U \to U$. See, for example,  \cite{Voisin-Some-Aspects, ElZein-Zucker}.

A more geometric way to define the morphism $\phi_X$ is in terms of algebraic cycles and Abel-Jacobi maps, which is what we use here. This is possible because the integral Hodge conjecture holds for degree $4$ Hodge class on $X$  \cite{Voisin-Some-Aspects, Zucker-HC}. It allows us to avoid, in the current presentation, defining the normal function associated with a cohomology class. The reader should keep in mind, however, that constructing an algebraic section of the intermediate Jacobian fibrations with a Hodge class on $X$ is a key ingredient in the proof of the Hodge conjecture of  \cite{Voisin-Some-Aspects, Zucker-HC}, so the short cut is only at the level of our presentation.

As  already mentioned, the integral Hodge conjecture holds for degree $4$ Hodge classes on $X$. In particular for every  class $\alpha \in H^{2,2}(X, \Z)$ there is an algebraic cycle $Z$ such that $[Z]=\alpha$. Let $V \subset \P^5$ be the open subset parametrizing smooth hyperplane sections of $X$ that do not contain any of the components of $Z$. If $\alpha$ is a primitive cohomology class, then for  $b=[Y_b] \in V$ the $1$-cycle $Z_b$ satisfies
\[
[Z_{Y_b}]=0 \,  \text{ in } \,H^4(Y_b, \Z) = \Z
\]
and hence determines a point $ \phi_{Y_b}(Z_b) \in \Jac(Y_b)$ in the intermediate Jacobian of $Y_b$.  By Griffiths \cite{Griffiths-periods-III} (see also  \cite[\S 7.2.1]{Voisin2}) the assignment
\[
\begin{aligned}
\sigma_Z: V & \longrightarrow  J_V,\\
 b & \longmapsto  \phi_{Y_b}(Z_b)
\end{aligned}
\]
defines a holomorphic section of the restriction of $J$ to $V \subset \P^5$. By \cite{Zucker-inventiones}, this section is, in fact, algebraic: indeed, consider a  Lefschetz pencil  $\Yp \to \P^1 $ of hyperplanes of $X$ with $\P^1 \subset V$ and with the property that none of the singular points of the members of the pencil are contained in $Z$. By \cite[(4.58)]{Zucker-inventiones} the restriction of $\sigma_Z$ to the non-empty open subset $V \cap \P^1$ of the pencil extends to a holomorphic function on all of $\P^1$ and is thus algebraic (see also \cite{ElZein-Zucker}). Since a holomorphic function that is algebraic in each variable is algebraic (see for example \cite[Thm. 5 Chap. IX]{Bochner-Martin}), it follows that $\sigma_Z$ actually defines a rational function on $\P^5$, i.e.
\[
\sigma_Z \in MW(\pi).
\]
Notice that the holomorphic section $\sigma_Z$ does not depend on the algebraic cycle representing $\alpha$. Indeed, since $\CH_0(X)=\Z$, by \cite[Thm 6.24]{Voisin-Chow} it follows that the cycle map $\CH^2(X) \to H^{2,2}(X,\Z)$ is injective. It follows that if $Z$ and $Z'$ are homologous, then they are  rationally equivalent in $X$ and hence so are their restrictions to a general smooth hyperplane section.
The conclusion of this discussion  is that the Abel-Jacobi map induces a well defined group homomorphism
\be \label{AbelJacobi}
\begin{aligned}
\phi_X: H^{2,2}(X, \Z)_0 & \longrightarrow MW(\pi).\\
 \alpha=[Z] & \longmapsto \sigma_\alpha:=\sigma_Z
 \end{aligned}
\ee

We prove injectivity of $\phi_X$ in \S \ref{sec injectivity} and surjectivity in \S \ref{sec surjectivity}. Since we will restrict to general pencils in $\P^5$, we start by recalling a few standard facts about Lefschetz pencils of cubic threefolds.

\subsection{Preliminaries on Lefschetz pencils}

We start by setting up the notation. Let $\P^1 \subset (\P^5)^\vee$ be a Lefschetz pencil with base locus a smooth cubic surface $\Sigma \subset X$. We have the following diagram
\[
\xymatrix{
& \Sigma \times \P^1 \ar[dl]_{p_1}  \ar@{^{(}->}[r]^i & \Yp \ar[dl]_p  \ar[dr]^q & \\
\Sigma \ar@{^{(}->}[r] & X &  & \P^1
}
\]
where $\Yp=Bl_\Sigma X$, where $q:\Yp \to \P^1$ is the fibration of threefolds, and where $i: \Sigma \times \P^1 \to \Yp$ is the inclusion of the exceptional divisor in $\Yp$.  Let $j: \Up \subset \P^1$ be the open subset parametrizing  smooth fibers.

The following Lemma is standard. We include a proof for lack of reference. 

\begin{lemma} \label{integral loc inv cycle}  The homology and cohomology groups of a cubic threefold which is smooth or has one $A_1$ singularity have no torsion. Moreover, using the notation as above, 
\[
R^1q_* \Z=0, \quad R^2q_* \Z=\Z, \quad R^3q_* \Z=j_* j^* R^3q_* \Z, \quad R^4q_* \Z=\Z.
\]
\end{lemma}
\begin{proof}
The  statement about the homology groups of a cubic threefold with at most an $A_1$ singularity follow from \cite[Example 5.3 and Theorem 2.1]{Dimca-homology-compositio}; using the universal coefficient theorem, the statement on the cohomology groups then follow. From loc. cit it also follows that $H^4(Y, \Z)=H_4(Y, \Z)^\vee=\Z$, and hence $R^4q_* \Z=\Z$ follows by proper base change.
The first two statements on the higher direct images follow from Lefschetz hyperplane section theorem. The third equality, which is also known as the ``local invariant cycle'' property, can be seen as follows (it is well known to hold with $\Q$-coefficents, we now show it with $\Z$-coefficients). By adjunction, there is a natural morphism
\[
\epsilon: R^3q_* \Z \to j_* j^* R^3q_* \Z
\]
which is an isomorphism over $U$. To show $\epsilon$ is an isomorphism over any point of $B:=\P^1\setminus \Up$ we  restrict, for every $b_0 \in B$, to a small disk $\Delta$ centered at $b_0$. Then $\epsilon$ is an isomorphism around $b_0$ if and only if the specialization morphism
\[
H^3(Y_{b_0}, \Z) \cong H^3(\Yp, \Z) \to H^3(Y_{b}, \Z)^{inv}=(j_* j^* R^3q_* \Z)_{b_0}
\]
is an isomorphism (cf. \cite[pg 439-440]{Steenbrink}), where $b \in \Delta \cap \Up$ and $H^3(Y_{b}, \Z)^{inv} \subset H^3(Y_{b}, \Z)$ are the local monodromy invariants. Let $\delta \in  H_3(Y_{b}, \Z)$ be the vanishing cycle of $\Yp_\Delta$. By the Picard--Lefschetz formula $H^3(Y_{b}, \Z)^{inv}=\Z \delta^\perp$, where $\perp$ is taken with respect to the intersection product (which is non--degenerate since $H^3(Y_{b}, \Z)$ is torsion free). By \cite[Cor. 2.17]{Voisin2}, there is a short exact sequence
\[
0 \to \Z \delta  \to H_3(Y_b, \Z) \to H_3(\Yp_\Delta, \Z)\cong H_3(Y_{b_0}, \Z) \to 0.
\]
where $0 \neq \delta \in H_3(Y_b, \Z)$ is the class of the vanishing cycle. 
Dualizing, we get a short exact sequence
\[
0 \to H^3(Y_{b_0}, \Z)\to H^3(Y_b,\Z) \to (\Z \delta)^\vee \to 0.
\] 
(recall  the absence of torsion in the homology groups of $Y_b$ and $Y_{b_0}$).
Using the isomorphism $H_3(Y_b,\Z) \cong H^3(Y_b, \Z)$ induced by Poincar\'e duality we make the identification $\Z \delta^\perp=\ker [H^3(Y_{b}, \Z)\to \Z \delta^\vee]=\im[H^3(Y_{b_0}, \Z)\to H^3(Y_b,\Z)]$.
\end{proof}

It is well known that for a Lefschetz pencil the Leray spectral sequence with coefficients in $\Q$ degenerates at $E_2$. For a Lefschetz pencil of cubic threefolds, this is true also for $\Z$ coefficients. Again, we include a proof for lack of reference.
For the whole family of smooth hyperplane sections of $X$ the Leray spectral sequence with integers coefficients does \emph{not} degenerate at $E_2$; this is the starting point of the construction of the non trivial $J_U$-torsor of \cite{Voisin-twisted} (cf. Remark \ref{twisted}).

\begin{lemma} \label{Leray filtration} Let $q: \Yp \to \P^1$ be as above. The Leray spectral sequence with $\Z$ coefficients degenerates at $E_2$. In particular, the Leray filtration on $H^4(\Yp, \Z)$ is given by:
\be \label{leray filtration}
\begin{aligned}
\Z =H^2(\P^1, R^2 f_*\Z) \subset L_1 \subset H^4(\Yp, \Z) \twoheadrightarrow H^0(\P^1, R^4 f_*\Z)=\Z\\
0 \to H^2(\P^1, R^2 f_*\Z) \to L_1 \stackrel{\gamma}{\to} H^1(\P^1, R^3 f_*\Z) \to 0.
\end{aligned}
\ee
\end{lemma}
\begin{proof} Because of the many vanishings in the $E_2$-page of the spectral sequence, the only map we need to show is trivial is $H^0(\P^1, R^4 q_*\Z) \to H^2(\P^1, R^3 q_*\Z)$. For this, it is enough to show that $H^4(Y_b,\Z) \to H^0(\P^1, R^4 q_*\Z)$ is surjective, which is clearly true since both groups are generated by the class of a line.
\end{proof}

Consider the decomposition
\be \label{blowup}
H^4(\Yp,\Z)=H^4(X,\Z) \oplus H^2(\Sigma,\Z) \oplus H^0(\Sigma, \Z),
\ee
given by the blowup formula.
The inclusion of the first summand is given by the pullback $p^*$; we freely omit the symbol $p^*$ when viewing $H^4(X,\Z) $ as a subspace of $H^4(\Yp,\Z)$.  The inclusion of the second factor is given by $
H^2(\Sigma,\Z) \ni C   \mapsto i_*(C \times \P^1) \in  H^4(\Yp,\Z)$. Finally,
 the inclusion of the last summand is given by $H^0(\Sigma,\Z)=H^0(\Sigma, \Z) \otimes H^2(\P^1,\Z) \ni [\Sigma]=[\Sigma \times p] \mapsto i_*([\Sigma \times p]) \in H^4(\Yp,\Z)$, where $p \in \P^1$ is a point.
We highlight the following results for later use.

\begin{lemma} \label{lemmakerAJ} There is a natural isomorphism $H^0(\Sigma, \Z)  \cong H^2(\P^1, R^2 q_* \Z)$ which allows the identification of the inclusion $H^0(\Sigma, \Z)\cong H^0(\Sigma, \Z)\otimes H^2( \P^1,  \Z)  \stackrel{i_*}{\to} H^4(\Yp,\Z)$ of (\ref{blowup}) with the inclusion $H^2(\P^1, R^2 q_* \Z) \to H^4(\Yp,\Z)$ induced by the Leray filtration (\ref{Leray filtration}).
\end{lemma}
\begin{proof}
The closed embedding $i: \Sigma \times \P^1 \hookrightarrow \Yp$ determines an isomorphism $ {p_2}_* \Z \cong R^2 q_* \Z$ of  constant local systems. Here $p_2: \Sigma \times \P^1 \to \P^1$ is the projection on the section factor. Since $H^2( \P^1, {p_2}_*  \Z) = H^0(\Sigma, \Z)  \otimes H^2(\P^1,\Z) $ the lemma follows.
\end{proof}

Via $p^*$,  we  make the identification $H^4(X,\Z)_0  \cong L_1 \cap H^4(X,\Z)$ and we set $ L_1^{2,2}=L_1 \cap H^{2,2}(\mc Y',\Z)$. Here $L_1 \subset H^4(\mc Y',\Z)$ denotes the second piece of the Leray filtration (cf. (\ref{leray filtration})).

\begin{cor} \label{periniettivita} The surjective  morphism $\gamma: L_1 \to H^1(\P^1, R^3 q_* \Z)$ of (\ref{leray filtration}) restricts to an injection
\[
\bar \gamma:L_1 \cap (H^{2,2}(X,\Z) \oplus H^2(\Sigma,\Z)) \cong L_1^{2,2} \slash \ker(\gamma) \to  H^1(\P^1, R^3 q_* \Z).
\]
\end{cor}
\begin{proof} From   lemmas  \ref{Leray filtration} and \ref{lemmakerAJ} above, it follows that $\ker (\gamma)=H^2(\P^1, R^2 q_* \Z)=H^0(\Sigma, \Z) $. Thus, by (\ref{blowup}), it follows that $H^4(X,\Z)\oplus H^2(\Sigma,\Z)) \cap \ker(\gamma)=\{0\}$. Since $H^0(\Sigma, \Z) \subset L_1$ and
\be \label{uffa}
L_1 \cap \Big(H^{2,2}(X,\Z) \oplus H^2(\Sigma,\Z) \oplus  H^0(\Sigma, \Z)\Big)=L_1 \cap \Big(H^{2,2}(X,\Z) \oplus H^2(\Sigma,\Z) \Big) \oplus  H^0(\Sigma, \Z)
\ee
the Corollary follows.
\end{proof}

\begin{lemma} \label{fiveterm}
The restriction morphism $H^1(\P^1, R^3 q_*\Z) \to H^1(\Up, R^3 {q_\Up}_*\Z)$ is injective.
\end{lemma}
\begin{proof}
The Leray spectral sequence for the open immersion $j: \Up \to \P^1$, applied to the sheaf $j^*R^3 q_*\Z= R^3 {q_\Up}_*\Z$, gives a $5$-term exact sequence starting with
\[
0 \to H^1(\P^1, j_* j^*R^3 q_*\Z) \to H^1(\Up, R^3 {q_\Up}_*\Z) \to \dots
\]
This concludes the proof, since by Lemma \ref{integral loc inv cycle}, $R^3 q_*\Z= j_* j^*R^3 q_*\Z$.
\end{proof}

\subsection{Injectivity of $\phi_X$.} \label{sec injectivity} The proof of injectivity uses the Hodge class of a normal function (cf. \cite{Zucker-inventiones} and \cite[\S 8.2.2]{Voisin2}).

For a pencil $\Yp \to \P^1$ as above, set
\[
H^{2,2}(\Yp, \Z)_0:=L_1^{2,2}=\ker[H^{2,2}(\Yp, \Z) \to H^0(\P^1, R^4 q_*\Z)]=L_1 \cap H^{2,2}(\Yp, \Z).
\]
and let
\[
\pi'=J' \to \P^1, \text{ and } J_\Up \to \Up,
\]
 be the restriction of the intermediate Jacobian fibration to $\P^1$ and to $\Up$. Choosing a set of generators for $H^{2,2}(X, \Z)_0$, let $\Yp \to \P^1$ be  a general enough pencil so that the restriction morphism 
\be \label{AJ for pencil}
\phi_X': H^{2,2}(X, \Z)_0 \to MW(\pi')
\ee
is well defined. Here, $MW(\pi')$ is the group of rational sections of $\pi'$. Similarly, we get a group homomorphism $\phi'_\Yp:H^{2,2}(\Yp, \Z)_0 \to MW(\pi')$.
Moreover, if $\alpha \in H^{2,2}(X,\Z)_0$ then
\[
\phi'_\Yp(p^*\alpha)=\phi'_X(\alpha) \in MW(\pi').
\]

Recall the Hodge class of a normal function (cf. \cite[8.2.2]{Voisin2}, \cite[(3.9)]{Zucker-inventiones})). 
Let $ \mc H^3= R^3{q_\Up}_* \Z \otimes_\Z \mc O_\Up$ be the Hodge bundle associated to the weight $3$ variation of Hodge structure of the pencil and let $F^* \mc H^3$ be the Hodge filtration. The sheaf $\mc J_\Up$ of holomorphic sections of the intermediate Jacobian fibration fits in the following exact sequence 
\[
0 \to R^3{\fp_\Up}_* \Z \to \mc H^3 \slash F^2 \mc H^3 \to \mc J_\Up \to 0.
\]
and the coboundary morphism
\[
\begin{aligned}
H^0(\Up, \mc J_\Up ) & \stackrel{cl}{\longrightarrow} H^1(\Up,R^3f_* \Z)\\
\nu & \longmapsto cl(\nu)
\end{aligned}
\]
associates to every holomorphic section $\nu$ of $J_\Up \to \Up$ a class $cl(\nu)$  in $H^1(\Up,R^3q_* \Z)$, called the Hodge class of $\nu$ (in the present context, this class is of Hodge type with respect to the Hodge structure on $H^1(\Up,R^3q_* \Z)$ induced from that on $H^4(\Yp, \Z)$ via the degeneracy of the Leray spectral sequence, see \cite[\S 3]{Zucker-inventiones}).

\begin{lemma} \label{injectivity} Let $\Yp \to \P^1$ be a general pencil. The homomorphism of (\ref{AJ for pencil})
\[
\phi'_X: H^{2,2}(X, \Z)_0 \stackrel{\beta}{\longrightarrow}  MW(\pi')  
\]
is injective.
\end{lemma}
\begin{proof} 
By \cite[ Prop. (3.9)]{Zucker-inventiones} (see also \cite[Lem 8.20]{Voisin2}), the following diagram is commutative 
\be \label{diagraminj}
\xymatrix{
H^{2,2}(X, \Z)_0 \ar[r]^{p^*} \ar[dr]_{\phi'_X} & H^{2,2}(\Yp, \Z)_0 \ar[d]_{\phi'_\Yp}  \ar[r]^\gamma  & H^1(\P^1, R^3 p_* \Z) \ar[d]^\varepsilon \\
& H^0(\Up, \mc J_\Up) \ar[r]_{cl} &H^1(\Up,R^3\fp_* \Z)
}
\ee
The map $\varepsilon$ is injective by Lemma \ref{fiveterm}, and $p^* \circ \gamma$ is injective by Lemma \ref{periniettivita}. Hence, $cl \circ \phi'_X$ is injective and thus so is $\phi'_X$.

\end{proof}

\subsection{Surjectivity of $\phi_X$.} \label{sec surjectivity}

There are three ingredients in the proof of surjectivity: the fact that $\rk MW(\pi)=\rk H^{2,2}(X, \Z)_0$ as proved in Corollary \ref{ranks}; the restriction,  once again, to Lefschetz pencils; the techniques used in \cite{Voisin-Some-Aspects,Zucker-HC} for the proof of the integral Hodge conjecture for cubic fourfolds. We remark that we use their argument in a slightly different way. To prove the Hodge conjecture one starts with a cohomology class, uses it to define a normal function, and then uses  the normal function to construct an algebraic cycle representing the cohomology class (possibly up to a multiple of a complete intersection surface).  See \cite{Voisin-Some-Aspects} for more details. Here we start with a rational section of the intermediate Jacobian fibration, we restrict to a general pencil, and use the same method of Voisin to construct an algebraic cycle inducing the section via the Abel-Jacobi map. Then we have to check that the cohomology class representing this cycle is primitive, that it is independent of the pencil, and that it induces, via $\phi_X$, the section we started from. 

Since by Corollary \ref{ranks} the cokernel of the injection $\phi_X: H^{2,2}(X, \Z)_0 \to MW(\pi)$ is finite,  for any $\sigma \in MW(\pi)$ there is an integer $N$ and a cohomology class $ \alpha \in H^{2,2}(X, \Z)_0$ such that
 \be \label{Nsigma}
\sigma_\alpha:=\phi_X(\alpha)=N \sigma.
 \ee

We will show, again using Lefschetz pencils, that given $\sigma$ and $\alpha$ as above, there exists a $\bar \beta' \in H^{2,2}(X, \Z)_0$ such that $\alpha=N \bar \beta'$. This will give the desired surjectivity. Before we do so, let us introduce some results that we will need.

For a general pencil $\Yp \to \P^1$, let
\[
 (J^T)' \to \P^1
\]
be the restriction of the  intermediate Jacobian fibration   $J^T \to \P^5$ of \cite{Voisin-twisted} (cf. Remark \ref{twisted}) to the pencil. For a conic $C \subset \Sigma$, consider the relative $1$-cycle of degree $2$ in $\Yp \to \P^1$ (any other degree $2$ relative $1$--cycle that comes from $\Sigma$ will do). This defines a section of $(J^T)' \to \P^1$, which trivializes the torsor $(J^T)_\Up$ inducing an isomorphism $J'_\Up \cong(J^T)_\Up$. It is easily seen that this extends to an isomorphism $t_C: J' \cong(J^T)'$ over $\P^1$. For any $\sigma' \in H^0(\Up, \mc J_\Up)$, we may consider the induced section
\[
(\sigma^T)':=t_C \circ \sigma' \in H^0(\Up, \mc J^T_\Up)
\]

The following result is proved in Voisin \cite{Voisin-Some-Aspects} (see also  \cite[(3.2)]{Zucker-HC}, where the result is proved over $\Q$).

\begin{prop} (\hspace{1sp}\cite[\S 2.3]{Voisin-Some-Aspects}) \label{lifting} For any  section $\sigma' \in MW(\pi') $

there is a relative $1$-cycle $Z$ on $ \Yp$ of degree $2$, such that the cohomology class
\[
\beta'=[Z]-[C \times \P^1] \in  H^{2,2}(\Yp, \Z)_0
\]
satisfies $\phi'_{\Yp}(\beta')=\sigma'$ in $MW(\pi') $.

\end{prop}
\begin{proof}
For the reader's convenience, we give a brief sketch of the argument. By a result of Markushevich-Tikhomirov \cite{Markushevich-Tikhomirov} and Druel \cite{Druel} there is a a relative birational morphism $c_2:  \Mp_\Up \to  J^T_\Up$, where $ \Mp_\Up \to \Up$ is the relative moduli space of sheaves on $\Yp_{\Up} \to \Up$ with  $c_1=0$ and $c_2=2\ell$. The morphism associates to every sheaf corresponding to a point in $ \Mp_\Up $ the Abel-Jacobi invariant of its second Chern class. Given a section $(\sigma^T)' \in H^0(\Up, \mc J^T_\Up)$ as above, Voisin uses $ \Mp_\Up \to \Up$ to construct a family $\mc C_\Up$ of degree $2$ curves in the fibers of $\Yp_\Up \to \Up$ with the property that for every $b \in \Up$, the curve $\mc C_b$ represents the $c_2$ of a sheaf over $(\sigma^T)'(b)$. By construction, letting $Z $ be the closure of $\mc C_\Up$ in $\Yp$ and setting $\beta':=[Z]-[C \times \P^1] \in  H^{2,2}(\Yp, \Z)_0$, we have $\phi'_{\Yp}(\beta')=\sigma' $ in $H^0(\Up, \mc J_\Up)$.
\end{proof}

Let $\sigma \in MW(\pi)$. For a general pencil $\P^1 \subset \P^5$, let $\sigma'=\sigma_{|\P^1}$ be the restriction of $\sigma $ to $\P^1$, and let $\beta'$ be as in the Proposition above so that $\phi'_{\Yp}(\beta')=\sigma'$. It is tempting to say that, via $\phi_X$, the class $\beta'$ induces $\sigma$  globally and not just on that pencil. This is indeed the case, though we first  need to check that $\beta'$ lies in the primitive cohomology of $X$ and that $\beta'$ is independent of the pencil as well as of the chosen isomorphism $t_C: J' \cong(J^T)' $. More precisely, we need to check that $\beta'$ induces $\sigma$ over an open subset of $\P^5$ and not just on the chosen pencil. Before checking this, we have the following Proposition.

Recall that we have set $H^{2,2}(\Yp,\Z)_0=L_1 \cap H^{2,2}(\Yp,\Z)$.

\begin{prop}(\hspace{1sp}\cite[(4.17)]{Zucker-inventiones}) \label{propaj} The Abel-Jacobi morphism $\phi_{\Yp}: H^{2,2}(\Yp, \Z)_0 \to MW(\pi') \subset H^0(\Up, \mc J_\Up)$ is surjective and defines an isomorphism
\[
\bar \phi_{\Yp}: L_1 \cap (H^{2,2}(X,\Z) \oplus H^2(\Sigma,\Z)) \to MW(\pi')
\]
\end{prop}
\begin{proof} By diagram (\ref{diagraminj}) and the fact that $\epsilon$ is injective, $\ker (\phi_{\Yp})=\ker \gamma$, which by Lemma \ref{Leray filtration} is equal to $H^0(\Sigma, \Z)$. Since $\phi_{\Yp}$ is surjective by the proposition above, the induced morphism $\bar \phi_{\Yp}: H^{2,2}(\Yp,\Z)_0 \slash H^0(\Sigma, \Z) \to MW(\pi')$ is an isomorphism. Finally, by (\ref{uffa}),
$H^{2,2}(\Yp,\Z)_0 \slash H^0(\Sigma, \Z) \cong L_1 \cap (H^{2,2}(X,\Z) \oplus H^2(\Sigma,\Z))$.
\end{proof}

We can now end the proof of  surjectivity:
For $\sigma \in MW(\pi)$, let $\alpha \in H^{2,2}(X,\Z)_0$ be as in (\ref{Nsigma}). Restricting to a pencil $\Yp \to \P^1$,  
set $\sigma'=\sigma_{|\P^1}$ and let $\beta'$ be as in Proposition \ref{lifting} such that $\phi_{\mc Y'}(\beta')=\sigma'$. 
Finally, let $\bar \beta' $  be the projection of $\beta'$ onto $L_1 \cap (H^{2,2}(X,\Z) \oplus H^2(\Sigma,\Z))$. Notice that abuse of notation, we are omitting $p^*$ from  the inclusion of $H^4(X,\Z)$ in $H^4(\mc Y',\Z)$ and we will write $\alpha$ instead of $p^* \alpha$.
We have,
\[
\phi_{\mc Y'}(\alpha)=(\phi_X(\alpha))_{|\P^1}=N \sigma'=N \phi_{\mc Y'}(\beta' )= \phi_{\mc Y'}(N \beta' ).
\]
By Proposition \ref{propaj}, $ \alpha=N \bar \beta' \in L_1 \cap (H^{2,2}(X,\Z) \oplus H^2(\Sigma,\Z))$. Since $ \alpha \in  H^{2,2}(X,\Z)_0  \subset  L_1 \cap \Big(H^{2,2}(X,\Z) \oplus H^2(\Sigma,\Z)\Big)$, it follows that $\bar \beta'$, too, has to lie in  $H^{2,2}(X,\Z)_0 \subset H^{2,2}(\Yp, \Z)$. Moreover, the class $ \bar \beta'$, which a priori depends on the chosen Lefschetz pencil is independent of the pencil. Set $\sigma_{\bar \beta'}=\phi_X(\bar \beta')$. Then, for  \emph{any} sufficiently general Lefschetz pencil $\P^1 \subset \P^5$ we have an equality of sections
\[
(\sigma_{\bar \beta'})_{|\P^1} =\sigma_{|\P^1},
\]
and hence the two rational sections $\sigma_{\bar \beta'}$ and $\sigma$ coincide. This ends the proof of surjectivity.

\appendix

\section{On the Beauville conjecture for  LSV varieties, by Claire Voisin} \label{appendix}

We explain in this appendix a consequence of Corollary \ref{J not JT}, on the following conjecture  made by Beauville in \cite{beau}.

\begin{conj} \label{conjbeau} Let $M$ be a projective  hyper-K\"{a}hler manifold. Any polynomial cohomological relation
$P(d_1,\ldots,d_r)=0$ in $H^*(M,\mathbb{Q})$, where $d_i$ are divisor classes on $M$, already holds in ${\rm CH}(M)$.
\end{conj}
Here ${\rm CH}(M)$ denotes the Chow groups of $M$ with rational coefficients.
Let now $M\rightarrow B$ be a projective hyper-K\"ahler manifold  of dimension $2n$ equipped with a Lagrangian fibration,
and let $L\in{\rm Pic}(M)={\rm NS}(M)$ be the Lagrangian class pulled back from $B$ (see \cite{matsu}). We have $q(L)=0$
by the Beauville-Fujiki relations, since $L^{2n}=0$. Let also $h\in{\rm Pic}(M)={\rm NS}(M)$ be the class of an ample divisor on $M$, so that
the intersection pairing $q$ restricted to $\langle L,\,h\rangle$ is nondegenerate by the Hodge index theorem.
The same argument as in \cite{bogover} shows that the polynomial cohomological relations between $L$ and $h$ are generated by the relations \begin{eqnarray}
\label{eqpourchlag} \alpha^{n+1}=0\,\,{\rm in}\,\,H^{2n+2}(M,\mathbb{Q})\,\,{\rm when}\,\,q(\alpha)=0, \text{ for } \alpha \in \langle L,\,h\rangle.
\end{eqnarray}
 Here we can restrict to rational cohomology classes because we know that there is an isotropic class in $\langle L,\,h\rangle$. We consider now more specifically  an  LSV variety $J$ (which is of dimension $10$, so $n=5$) which is constructed  in \cite{LSV} as a Lagrangian fibration over $\mathbb{P}^5$. The Picard group of a very general such variety is a rank $2$ lattice which contains as above the  Lagrangian class $L$ and an ample class, but we take as a basis
 the classes $L,\,\Theta$, where $\Theta$ was introduced in \cite{LSV} and is studied in the present paper.
 Riess proved in \cite{riess} that a hyper-K\"ahler manifold $M$ which has a Lagrangian fibration and satisfies the ``RLF conjecture'' characterizing classes associated to  Lagrangian fibrations, satisfies Beauville's conjecture. However we do not know  that the LSV vareties satisfy the RLF conjecture.
We  prove here the following result.
\begin{thm}\label{theoappendix}  The relations (\ref{eqpourchlag}) hold in ${\rm CH}(J)$ for the lattice $\langle L,\,\Theta\rangle$ of an  LSV variety $J$. Conjecture \ref{conjbeau} is thus satisfied by a LSV variety with Picard number $2$.
\end{thm}
\begin{proof} There are (up to multiples) exactly  two classes $L$ and $L'$ in $\langle L,\,\Theta\rangle$ satisfying
$q(L)=0,\,q(L')=0$. Obviously $L^6=0$ in ${\rm CH}(J)$ since $L$ comes from the base which is of dimension $5$, so we only have to prove
that ${L'}^6=0$ in ${\rm CH}(J)$. We use Riess' argument in \cite{riess}, however in a different way.
As a consequence of  fundamental results of Huybrechts in \cite{Huybrechts-kahler-cone}, Riess proved the following:

\begin{thm} \cite[Theorem 3.3]{riess} Let $K$ be an isotropic class on a projective hyper-K\"ahler manifold $M$ of dimension $2n$. Then there exists
a cycle $\Gamma\in{\rm CH}^{2n}(M\times M)$ such that $\Gamma^*$ acts as an automorphism of ${\rm CH}(M)$ preserving
  the intersection product, the action of $\Gamma^*$ on $H^2(M)$ preserves the Beauville-Bogomolov form  $q_M$, and
$\Gamma^*K$ belongs to the boundary of the birational K\"ahler cone of $M$.
\end{thm}
Here the birational K\"ahler cone of $M$ is defined as the union of the  K\"ahler cones
of hyper-K\"ahler manifolds $M'$ bimeromorphic to $M$ (the bimeromorphic map $M'\dashrightarrow M$ inducing an isomorphism on $H^2$).
We apply this theorem to our class $L'$ on $J$ and thus get a correspondence  $\Gamma$ as above. The class $\Gamma^*L'$ is an isotropic class, hence it must be proportional to either $L'$ or $L$. Furthermore, it belongs to
the boundary of the birational K\"ahler cone. We now have
\begin{lemma} \label{le04}The class $L'$ does not belong to the boundary of the birational K\"ahler cone.
\end{lemma}
\begin{proof} This is proved in Corollary \ref{J not JT} of the present paper.
\end{proof}
By Lemma \ref{le04}, we conclude that $\Gamma^* L'$ is proportional to $L$. As $L^6=0$ in ${\rm CH}^6(J)$ and $\Gamma^*$ is an automorphism of ${\rm CH}(J)$ preserving the intersection product, we conclude that ${L'}^6=0$ in ${\rm CH}^6(J)$.
\end{proof}
If we consider  the case of Picard rank $3$, where the Picard lattice $N$ of $J$ is generated by three classes $L,\,\Theta,\,D$, with $q(L,D)=0,\,q(\Theta,D)=0$
there are now, according to \cite{bogover}, $13$ degree $6$ cohomological  relations between $L,\,\Theta$ and $D$, generated by
the classes $\alpha^6\in S^6N\subset S^6H^2(J,\mathbb{Q})$, where $\alpha$ belongs to the conic $q(\alpha)=0$.
Among these relations, two of them, namely those involving only $L$ and $\Theta$, are established in ${\rm CH}(J)$ by Theorem \ref{theoappendix}.
We also have the relations
\begin{eqnarray}\label{eqtangent} L^5D=0\,\,{\rm  and }\,\,{L'}^5D=0\,\,{\rm in}\,\, H^{12}(J,\mathbb{Q}),
\end{eqnarray}
 which
are obtained by differentiating
the relation (\ref{eqpourchlag}) at $ \alpha=L$ or $\alpha=L'$  in the direction given by $D$, (which is tangent to the conic at these points since $q(D,L)=0, \, q(D,L')=0$).
We prove the following:
\begin{thm} \label{theoapp2} The relations (\ref{eqtangent}) are satisfied in ${\rm CH}^6(J)$.
\end{thm}
\begin{proof}  The first relation is proved by applying the  following result from
\cite{voisintorsion}, which works in a more general context and needs  a mild assumption on the infinitesimal variation of Hodge structure of a family of abelian varieties at the   generic point of the base.   Let more generally $M\rightarrow B$ be a fibration into abelian varieties and let $A\in {\rm Pic}\,M$ be a line bundle
whose restriction to the general fiber $M_b$ is topologically trivial.
\begin{prop}  Assume that at the generic point $t\in B$,
there exists  a class $\alpha\in H^{1,0}(M_b)$ such that $\overline{\nabla}(\alpha): T_ {B,b}\rightarrow H^{0,1}(M_b)$
is surjective. Then there exists a point $b\in B$ such that $M_b$ is smooth and $A_{\mid M_b}$ is a torsion line bundle.

If all fibers $M_b$ have the same class  $F$  in ${\rm CH}(M)$, it thus follows that $F.A=0$ in ${\rm CH}(M)$.
\end{prop}
Coming back to our situation, we have to check that the assumption on the infinitesimal variation of Hodge structures
is satisfied in our situation. Let $J$ be the  LSV variety of a cubic fourfold $X$. The infinitesimal variation of Hodge structure for the fibers of the lagrangian fibration $J\rightarrow (\mathbb{P}^5)^\vee$ is thus canonically isomorphic to the variation of Hodge structure on the $H^3$ of the hyperplane section $X_H\subset X$.
If $Y $ is a smooth cubic threefold in $\mathbb{P}^4$ defined by an equation
$f=0$, Griffiths theory of IVHS of hypersurfaces says that there isomorphisms
$$H^{2,1}(Y)\cong R^1_f,\,\, H^{1,2}(Y)\cong R^4_f,$$
such that the infinitesimal variation of Hodge structure on $H^3(Y,\mathbb{C})$ is given (using the identification
$R^3_f\cong H^1(Y,T_{Y})$) by the multiplication map
$$R^3_f\rightarrow {\rm Hom}\,(R^1_f,R^4_f).$$
We now consider the case where $Y$ is a hyperplane section $X_H$, defined by a linear equation $H$, of the cubic fourfold  $X$.  It is immediate to see that the inclusions $X_H\subset X\subset \mathbb{P}^5$ determine a quadratic polynomial $Q_{X,H}\in R^2_f$ such that
the natural map $\rho: H^0(X_H,\mathcal{O}_{X_H}(1))\rightarrow R^3_f$, defined as the first order classifying map for the
deformations of $X_H$ in $X$, is given by multiplication by $Q_{X,H}$.
Combining these facts, we conclude that the
desired infinitesimal criterion for the fibration $J\rightarrow (\mathbb{P}^5)^\vee$ holds if
there exist a smooth hyperplane section $X_H\subset X$ and a linear form $x\in H^0(X_H,\mathcal{O}_{X_H}(1))=R^1_f$
 such that, with the above notation, the product map
  $$
xQ_{X,H}:R^1_f\rightarrow R^4_f
$$
by $ xQ_{X,H}$
is an isomorphism.
It is quite easy to show that the existence of such a hyperplane section is satisfied by $X$ in codimension $1$ in the moduli space of cubic fourfolds,
hence at the generic point of any Hodge locus in this moduli space, or equivalently any Noether-Lefschetz locus for the corresponding LSV variety $J$.
The relation $L^5D=0$ in ${\rm CH}^6(J)
$ is thus satisfied at the generic point of the deformation locus of  $J$ preserving the Hodge class $D$, hence everywhere by specialization.

To conclude the proof of Theorem \ref{theoapp2}, we have to prove the
relation $ {L'}^5D=0$ in ${\rm CH}^6(J)$. This follows however from the relation $L^5D=0$ in ${\rm CH}^6(J)
$ by the same argument as in the proof of Theorem \ref{theoappendix}, using the specialization of the cycle $\Gamma$
and observing that $\Gamma^*$ acts by $\pm1$ on $H^2(J,\mathbb{Q})^{\perp\langle L,\Theta\rangle}$, hence on $D$.
\end{proof}


\bibliographystyle{alpha}

\bibliography{bibliographyBGINTJAC.bib}

\end{document}